\documentclass[]{siamart171218}

\input{ex_shared.doc}

\ifpdf
\hypersetup{
  pdftitle={\TheTitle},
  pdfauthor={\TheAuthors}
}
\fi

\newcommand{\Ve}{V_{\ensuremath{\mathop{\mathrm{eff}}}}}
\newcommand{\We}{W_{\ensuremath{\mathop{\mathrm{eff}}}}}
\newcommand{\cVe}{\cV_{\ensuremath{\mathop{\mathrm{eff}}}}}

\begin{document}
\maketitle

\begin{abstract}
	In this work, we investigate a model order reduction scheme for polynomial parametric systems. We begin with defining the generalized multivariate transfer functions for the system. Based on this, we aim at constructing a reduced-order system, interpolating the defined generalized transfer functions at a given set of  interpolation points. Furthermore, we provide a method, inspired by the Loewner approach for linear and (quadratic-)bilinear systems, to determine a good-quality reduced-order system in an automatic way. We also discuss the computational issues related to the proposed method and a potential application of CUR matrix approximation in order to further speed-up simulations of reduced-order systems. We test the efficiency of the proposed methods via several numerical examples. 
\end{abstract}

\begin{keywords}
Model order reduction, interpolation, tensor algebra, matricization, polynomial dynamical systems, parametric systems, transfer functions.
\end{keywords}

\begin{AMS}
15A69, 34C20, 41A05, 49M05, 93A15, 93C10, 93C15
\end{AMS}

\section{Introduction}   
An accurate solution of time-dependent partial differential equations (PDEs), or  ordinary differential equations (ODEs), or a combination of both requires a fine spatial discretization of the governing equations. This leads to a large number of equations, thus a high-dimensional system. This inevitably imposes a huge burden on computational resources, and more often than not, it is \emph{almost} impossible to make use of such high-dimensional systems in engineering problem like, \emph{e.g.}, optimization or control. A way to resolve this issue is to construct a reduced-order system or a low-dimensional model, replicating the important dynamics of the original system. 

In this paper, we focus on  parametric polynomial systems of the form:
\begin{equation}
  \begin{aligned}
      E(\bp)\dot{x}(t,\bp) &= A(\bp)x(t,\bp) + \sum_{\xi = 2}^d H_{\xi}(\bp)\kronF{x}{$\xi$}(t,\bp) + \sum_{\eta = 1}^d N_{\eta}(\bp)\left(u(t)\otimes x^{\circled{\tiny {$\eta$}}}(t,\bp)\right) \\
      &\hspace{5.4cm} + B(\bp)u(t), \quad x(0,\bp) = 0,\\
      y(t,\bp) &= C(\bp)x(t,\bp),
  \end{aligned}\label{eq:para_poly_sys}
\end{equation}
where $E(\bp), A(\bp) \in \Rnn$, $B(\bp) \in \Rnm$, $C(\bp) \in \Rqn$, $H_\xi(\bp) \in \R^{n\times n^\xi}$, $\xi \in \{2,\ldots,d\}$, $N_\eta(\bp) \in \R^{n\times m\cdot n^\eta}$, $\eta \in \{1,\ldots,d\}$; the state, input and output vectors are $x(t) \in \Rn$, $u(t) \in \Rm$ and $y(t) \in \Rq$, respectively;  $x^{\circled{\tiny {$\xi$}}} := \underbrace{x\otimes \cdots\otimes x}_{\xi-\text{times}}$ and the parameter vector is denoted $\bp \in \R^{n_p}$.   Since the system \eqref{eq:para_poly_sys} has polynomial terms of the order up to $d$, we refer to it as a \emph{$d$-th order polynomial system}. The system \eqref{eq:para_poly_sys} lies in $n$-dimensional Euclidean  space and generally, $n$ is in $\cO\left(10^5\right)-\cO\left(10^6\right)$. Due to the computational burden mentioned above, we seek to construct a low-dimensional system, having the same structure as \eqref{eq:para_poly_sys}, which captures the dynamics of the original system \eqref{eq:para_poly_sys} for any given input $u(t)$ and parameter $\bp$ in a desired domain. 
 
Many of the widely used methods in model order reduction (MOR) to construct low-dimensional models for (parametric) nonlinear systems are based on \emph{snapshots}. This means that the state vector $x$ at time $t$ needs to be evaluated for a given input and parameter.  In this category, proper orthogonal decomposition is arguably the most favored method. This relies on determining the dominant subspace for the state vectors through singular value decomposition (SVD) of the collected snapshots, which is generally followed by computing a reduced-order system via Galerkin projection. For more details, we refer to~\cite{GubischV17mor}. For nonlinear systems, it is often combined with hyper-reduction methods such as EIM~\cite{morBarMNetal04} and DEIM~\cite{morChaS10} to further reduce computational costs related to the reduced nonlinear terms. Another widely known method in this category is the \emph{trajectory piecewise linear} method, in which a nonlinear system is approximated by a weighted sum of linearized systems (linearized along the trajectory). Then, each linear system is reduced using popular methods for linear systems such as balanced truncation or iterative methods, see, e.g.~\cite{morAnt05,morBenCOetal17,morBenMS05,morGugAB08}. Moreover, reduced basis methods, which are also snapshots-based methods, have been successfully applied to several nonlinear parametric systems, see, e.g., \cite{morQuaMN16}. Although these methods have been very successful, they share a common drawback of being dependent on snapshots, or in other words, simulations for given inputs and parameters. Hence, it may become  harder to obtain a reduced-order system to use e.g., in control. 

In this work, we rather focus on MOR methods, allowing us to determine reduced-order systems without any prior knowledge of inputs. There are, broadly speaking, two types of such methods, namely interpolation-based approaches and balanced truncation.  Recently, there have been significant efforts to extend these methods from linear to special classes of non-parametric polynomial systems, namely bilinear systems, and quadratic-bilinear systems, see, e.g., \cite{antoulas2016model,morBenB12b,morBenB15,morBenG17,morBenGG18,gosea2018data}. For parametric nonlinear systems, there has been a very recent work for bilinear parametric systems \cite{rodriguez2018interpolatory}, where the construction of an interpolating reduced system has been proposed for a  given set of interpolation points, and such a problem for quadratic-bilinear parametric systems still remains to be studied.

In this paper, we investigate an interpolation-based MOR scheme to obtain a reduced-order system for the parametric system \eqref{eq:para_poly_sys}. For this, we first define generalized transfer functions for the system \eqref{eq:para_poly_sys}. Based on this, we aim at constructing a reduced-order system such that its generalized transfer functions interpolate those of the original system at a given set of interpolation points for the frequency and parameters. Furthermore, we propose a scheme, inspired by the Loewner approach for linear and (quadratic-)bilinear systems \cite{antoulas2016model,gosea2018data}, thus leading to an algorithm that allows us to construct a good quality reduced-order system in an automatic fashion. Furthermore, we discuss related computational aspects and an application of pseudo-skeletal matrix approximation, the so-called CUR, to further reduce the computational complexity related to the reduced nonlinear terms. 

The remaining structure of the paper is as follows. In the following section, we discuss polynomialization of nonlinear systems and recap some basic concepts from tensor algebra. In \Cref{sec:TF_interpolation}, we present the generalized transfer functions corresponding to \eqref{eq:para_poly_sys} for a fixed parameter vector and discuss the construction of an interpolating reduced-order system using  Petrov-Galerkin projection. Based on this, we propose an algorithm which allows us to determine a good quality reduced-order system in an automatic fashion. In \Cref{sec:computationalCUR}, we discuss the related computational aspects and investigate an application of CUR matrix approximation to further reduce the complexity of the reduced nonlinear terms. In the subsequent section, we extend the proposed method to polynomial parametric systems. In \Cref{sec:Numerics}, we illustrate the efficiency of the proposed algorithms by means of two benchmark problems and their variants. We conclude the paper with a summary of our contributions and future perspectives. 

We make use of the following notation in the paper:
\begin{itemize}
	\item \texttt{orth}():  it returns an orthonormal basis of a given matrix.
	\item The Hadamard product and Kronecker product are denoted by `$\circ$' and `$\otimes$', respectively.
	\item Using \matlab~notation, $A(:,1 \mathbin{:} r)$ denotes the first $r$ columns of the matrix $A$, and $A(i,j)$ is the $(i,j)$th element of the matrix $A$.
	\item $I_m$ is the identity matrix of size $m\times m$.
	\item $\cV^{\circled{\tiny{$\xi$}}}$ is a short-hand notation for $\underbrace{\cV\otimes \cdots\otimes \cV}_{\xi-\text{times}}$, where $\cV$ is a vector/matrix.
\end{itemize}

\section{Polynomialization of Nonlinear Systems and Tensor Algebra}
In this section, we recap two topics. We begin with the polynomialization of nonlinear systems. 
\subsection{Polynomialization of nonlinear systems} \label{subsec:polynomialization}
A class of nonlinear systems, containing nonlinear terms such as exponential, trigonometric, rational, can be rewritten as a polynomial system \eqref{eq:para_poly_sys}, by introducing some auxiliary variables. This process is very closely related to the McCormick relaxation, used in nonconvex optimization \cite{mccormick1976computability}.  In the recent past, due to advances in the methodologies for MOR for quadratic-bilinear (QB) systems, there has been a substantial focus on rewriting a nonlinear system into the QB form. However, in the subsection, we will illustrate with an example how a polynomialization of a nonlinear system is  done by introducing less auxiliary variables as compared to its quadratic-bilinearization. 
\subsection*{An illustrative example} Let us consider the following  one-dimensional nonlinear ODE:
\begin{subequations}\label{eq:artificialExa}
\begin{align}
	\dot{x}(t) &=  -x(t) - x^3(t)  \cdot e^{-x(t)} + u(t),\\
	y(t) &= x(t).
\end{align}
\end{subequations}
Now, we seek to rewrite the system \eqref{eq:artificialExa} as a polynomial system via polynomialization. For this, we introduce an auxiliary variable as $z(t) := e^{-x(t)}$ and derive the corresponding differential equation. That is
\begin{equation*}
\dot{z}(t) = -e^{-x(t)}\dot{x}(t) = -z(t)\left(-x(t) -x^3(t) z(t) + u(t)\right).
\end{equation*}
Thus, we can equivalently write the input-output system \eqref{eq:artificialExa} as follows:
\begin{align*}	
	\bbm \dot{x}(t) \\ \dot{z}(t) \ebm &= \bbm -x(t)  \\ 0 \ebm  + \bbm 0\\ x(t)z(t) \ebm - \bbm  x^3(t) z(t) \\  0 \ebm + \bbm 0\\   x^3(t)z^2(t) \ebm - \bbm 0 \\ z(t)\ebm u(t) + \bbm u(t) \\ 0\ebm,\\
	y(t) &= \bbm 1&0\ebm \bbm x(t) \\z(t)\ebm.
\end{align*}
As can be seen, the system  \eqref{eq:artificialExa}, which has cubic exponential nonlinearity, can be rewritten into a polynomial system~\eqref{eq:poly_sys} of order $5$  by introducing a single variable. However, if one aims at rewriting the system into the QB form, then we need to introduce at least $3$ more auxiliary variables, which somehow also makes the resulting system even more complicated, hence, also impeding its model reduction. Therefore, it is  advantageous to  work with the polynomial system of order $5$ and thus, we aim at reducing polynomial systems with MOR schemes for polynomial systems. 

Furthermore, we emphasis that it still remains an open problem how many auxiliary variables are required minimally in order to rewrite a smooth nonlinear system into a polynomial system, which demands further research. However, we mention that there has some been some initial work in \cite{morGu11} in this direction. 
\subsection{Tensor Algebra} As the nonlinear part of the considered systems are written in Kronecker (tensor) format, we will need a number of tensor based calculations in the reminder of this paper. We will review or introduce the necessary concepts in this subsection. 
A tensor is a multidimensional or an $N$-way array. An $N$th-order tensor $\boldsymbol \cX \in \R^{n_1\times \cdots\times n_N}$ is an $N$-dimensional array with entries $\cX(i_1,\ldots,i_N)$ $ \in \R$, where $i_j \in \{1,\dots,n_j\}, j \in \{1,\ldots,N\}$. For illustration, in \Cref{fig:3rdOrderTensor}, we present an illustration of third-order tensor. 
\begin{figure}[!tb]
	\centering
	\includegraphics{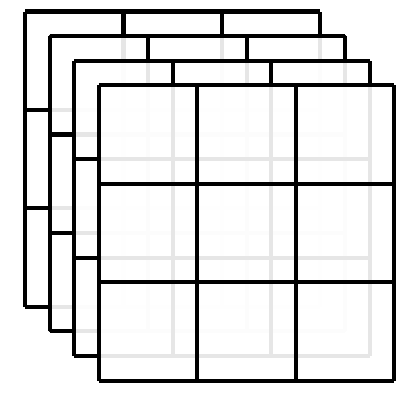}
	\caption{Illustration of a third-order tensor.}
	\label{fig:3rdOrderTensor}
\end{figure}
 An important concept of a tensor is the so-called  \emph{matricization}.  This allows us to unfold a tensor into a matrix, which plays a crucial role in tensor computations. For an $N$th order tensor, there are $N$ different ways to unfold as a matrix.  In the following, we define mode-$n$ matricization of a tensor $ \bten{X}$. 
 \begin{definition}[e.g., \cite{kolda2006multilinear}]
	The mode-$n$ matricization of a tensor $\bten{X} \in  \R^{n_1\times \cdots\times n_N}$, denoted by $\boldsymbol X_{(n)}$, satisfies the following mapping:
	\begin{equation*}
		\boldsymbol{X}_{(n)}(i_n,j) = \bten{X}(i_1,\ldots,i_N),
	\end{equation*}
	where $j = 1+ \sum\limits_{k=1,k\neq n}^N(i_k-1)J_k$ with $J_k = \prod\limits_{m=1,m\neq n}^{k-1}n_m$. 
\end{definition}

Like matrix-vector and matrix-matrix products, tensor-tensor, tensor-matrix and tensor-vector products can be defined; however, the notation becomes quite cumbersome.  In the following, we present a connection between the mode-$n$ matricization and Kronecker products. For this, we define the following tensor-matrix product:
\begin{equation*}
 \bten{Y} = \bten{X} \times_1 \mathrm A^{(1)} \times_2 \mathrm A^{(2)} \cdots \times_N\mathrm A^{(N)},
\end{equation*}
where $\mathrm A^{(l)} \in \R^{J_l\times n_l }$. Then, we have the following relation:
\begin{equation}\label{eq:matricizationRelationMatrix}
 Y_{(n)} =\mathrm A^{(n)}X_{(n)}\left(\mathrm A^{(N)}\otimes \cdots\otimes\mathrm A^{(n+1)} \otimes\mathrm A^{(n-1)} \otimes \mathrm A^{(1)} \right)^T.
\end{equation}
Of particular interest of the paper, we explicitly note down the results for tensor-vector products as well. For this, let us define the following product:
\begin{equation}
\bten{Y} = \bten{X} \bar\times_1\mathrm a_1\bar\times_2\cdots \bar\times_M\mathrm a_M,
\end{equation}
where $M \leq N$ and $\mathrm a_m\in \R^{j_m}$, $m \in \{1,\ldots,M\}$. Then, using \cite[Prop. 3.7]{kolda2006multilinear}, we define
\begin{equation}\label{eq:matricizationRelation}
\boldsymbol{Y}_{(n)} = a_n^T\boldsymbol{X}_{(n)}\left( \mathrm a_M\otimes\cdots\otimes \mathrm a_{n+1}\otimes \mathrm  a_{n-1}\otimes \cdots \otimes \mathrm a_1 \right), \quad n \in \{1,\ldots,M\}.
\end{equation}
Furthermore, we consider a special case, which is very useful later in the paper, that is when $N = M$. In this case, the quantity 
\begin{equation*}
\bten{X} \bar\times_1\mathrm a_1\bar\times_2\cdots \bar\times_N\mathrm a_N =: \Xi
\end{equation*}
 is a scalar. Hence, using \eqref{eq:matricizationRelation}, we obtain the following relation:
\begin{equation}
\begin{aligned}
&a_1^T\boldsymbol{X}_{(1)}\left( \mathrm a_N\otimes\cdots\otimes  \mathrm a_2 \right) = a_2^T\boldsymbol{X}_{(2)}\left( \mathrm a_N\otimes\cdots\otimes a_3 \otimes a_1 \right) = \cdots\\ &\hspace{7.5cm} = a_N^T\boldsymbol{X}_{(N)}\left( \mathrm a_{N-1}\otimes\cdots\otimes \mathrm a_1 \right).
\end{aligned}
\end{equation}
For further details on tensor concepts such as tensor-matrix, and tensor-vector multiplications, and matricization, we refer readers to \cite{kolda2006multilinear} and references therein.
\section{Construction of Interpolating Reduced-Order Systems}\label{sec:TF_interpolation}
In this section, we discuss the construction of interpolating reduced-order systems. For simplicity, we begin with non-parametric polynomial systems as follows:
\begin{equation}\label{eq:poly_sys}
 \begin{aligned}
 \dot{x}(t) &= Ax(t) + \sum_{\xi = 2}^d H_{\xi}x^{\circled{\tiny {$\xi$}}}(t)+ \sum_{\eta = 1}^d N_{\eta}\left(u(t)\otimes x^{\circled{\tiny {$\eta$}}}(t)\right) + Bu(t), \quad x(0) = 0,\\
 y(t) &= Cx(t),
 \end{aligned}
\end{equation}
where $x(t)\in \Rn$, $u(t)\in \Rm$ and $y(t) \in \Rq$ are state, input and output vectors respectively, and all other matrices are constants and are of appropriate size. Moreover, the system \eqref{eq:poly_sys} is referred to as a single-input single-output (SISO) system when $q=m=1$; otherwise, it is referred to as a multi-input multi-output system~(MIMO).
\subsection{Reduced-order modeling for SISO systems}
We begin with considering SISO polynomial systems \eqref{eq:poly_sys}. As a first step towards developing a MOR scheme for the system, we aim at defining  the generalized multivariate transfer functions. Following the steps as shown in \cite{morBenGG18} for QB systems, we write the Volterra series corresponding to the system \eqref{eq:poly_sys} as follows:
\begin{multline}\label{eq:system_state}
    x(t) = \int_0^t e^{A\sigma_1} Bu(t_{\sigma_1})d\sigma_1 + \sum_{\xi = 2}^d \int_0^t e^{A\sigma_1} H_\xi x^{\circled{\tiny {$\xi$}}} (t_{\sigma_1}) d\sigma_1 \\ 
    + \int_0^t \sum_{\eta = 1}^d e^{A\sigma_1} N_{\eta} \left.x^{\circled{\tiny {$\xi$}}}(t_{\sigma_1})\right.  u(t_{\sigma_1}) d\sigma_1,
 \end{multline}
where $t_{\sigma_1} := t-\sigma_1$. The above equation also allows us to express $x(t_{\sigma_1})$ as follows:
\begin{multline}\label{eq:expressionxsigma1}
    x(t_{\sigma_1}) = \int_0^{t_{\sigma_1}} e^{A\sigma_2} Bu(t_{\sigma_{1}}{-}\sigma_2)d\sigma_2 + \sum_{\xi = 2}^d\int_0^{t_{\sigma_1}} e^{A\sigma_2} H_\xi x^{\circled{\tiny {$\xi$}}}(t_{\sigma_1}{-}\sigma_2) d\sigma_2 \\
     + \sum_{\eta = 1}^d \int_0^{t_{\sigma_1}} e^{At_{\sigma_2}} N_\eta\left. x^{\circled{\tiny {$\eta$}}}(t_{\sigma_1}{-}\sigma_2) \right. u(t_{\sigma_1}{-}\sigma_2) d\sigma_2.
 \end{multline}
Substituting the expression in \eqref{eq:expressionxsigma1} for $x(t_{\sigma_1})$ in \eqref{eq:system_state} and multiplying by $C$ yields
\begin{align*}\allowdisplaybreaks
 y(t) &= \int_0^t Ce^{A\sigma_1} Bu(t_{\sigma_1})d\sigma_1 \\
 &\qquad+ \sum_{\xi = 2}^d \int_0^t \underbrace{\int_0^{t_{\sigma_1}}{\cdots} \int_0^{t_{\sigma_1}}}_{\xi-\text{times}} Ce^{A\sigma_1}  H_\xi \left(e^{A\sigma_2}B \otimes \cdots \otimes e^{A\sigma_{\xi+1}}B \right)d\sigma_1 d\sigma_2\cdots d\sigma_{\xi+1} \\
 &\qquad + \sum_{\eta = 1}^d \int_0^t \underbrace{\int_0^{t_{\sigma_1}}\cdots \int_0^{t_{\sigma_1}}}_{\eta-\text{times}} Ce^{A\sigma_1} N_\eta\left(e^{A\sigma_2}B \otimes \cdots \otimes e^{A\sigma_{\eta+1}}B \right)\\
 &\qquad \times \left(u(t_{\sigma_1})  u(t_{\sigma_1}-\sigma_2) \cdots  u(t_{\sigma_1}-\sigma_{\eta+1}) \right) d\sigma_1 d\sigma_2\cdots d\sigma_{\eta+1} + \cdots.
 \end{align*}
As the above Volterra series, corresponding to the system \eqref{eq:poly_sys}, is cumbersome and contains infinitely many terms, we consider only the leading \emph{kernels} of the series, which are as follows:
\begin{subequations}
\begin{align}
 f_L(t_1) &:= Ce^{At_1} B,\\
 f^{(\xi)}_H(t_1,\ldots,t_{\xi+1}) &:= Ce^{At_1}  H_\xi \left(e^{At_2}B \otimes \cdots \otimes e^{At_{\xi+1}}B\right), \\
 f^{(\eta)}_N(t_1,\ldots,t_{\eta+1}) &:= Ce^{At_1} N_\eta\left(e^{At_2}B \otimes \cdots \otimes e^{At_{\eta+1}}B \right),
\end{align}
\end{subequations}
where $\xi \in \{2,\ldots,d\}$ and $\eta \in \{1,\ldots,d\}$. Furthermore, taking the multivariate Laplace transform (see, e.g., \cite{rugh1981nonlinear}) of the above kernels, we get the frequency-domain representations of the kernels as follows:
\begin{subequations}\label{eq:general_TF}
\begin{align}
 \bF_L(s_1) &:= \cL(f_L) =  C\Phi(s_1)B,\\
 \bF^{(\xi)}_H(s_1,\ldots,s_{\xi+1}) &:= \cL(f^{(\xi)}_H) =  C\Phi(s_{\xi+1}) H_\xi \left(\Phi(s_\xi)B \otimes \cdots \otimes \Phi(s_{1})B\right),\\
 \bF_N^{(\eta)}(s_1,\ldots, s_{\eta+1}) &:= \cL(f_N^{(\eta)}) =  C \Phi(s_{\eta+1}) N_\eta\left(\Phi(s_{\eta})B \otimes \cdots \otimes \Phi(s_{1})B \right),
\end{align}
\end{subequations}
where $\Phi(s) = (sI_n{-}A)^{-1}$ is the so-called state transition matrix, and $\cL(\cdot)$ denotes the multivariate Laplace transform. In the above, we have assumed that the mass matrix in front of $\dot x(t)$ in \eqref{eq:poly_sys} is $E = I_n$; however, one can also perform the above algebra to derive the multivariate transfer function for $E\neq I_n$ but invertible. In this case, we can also obtain the multivariate transfer functions as in \eqref{eq:general_TF}, where the matrix $\Phi(s)$ will be $(sE{-}A)^{-1}$ instead of $(sI{-}A)^{-1}$. In the rest of the paper, we assume that the matrix $E$ is an invertible matrix. We aim at constructing reduced-order systems, having a similar structure as in \eqref{eq:poly_sys}, as follows:
\begin{equation}\label{eq:red_poly_sys}
\begin{aligned}
\hE\dot{\hx}(t) &= \hA\hx(t) + \sum_{\xi = 2}^d \hH_{\xi}\hx^{\circled{\tiny {$\xi$}}}(t)+ \sum_{\eta = 1}^d \hN_{\eta}\left(u(t)\otimes \hx^{\circled{\tiny {$\eta$}}}(t)\right) + \hB u(t), \quad \hx(0) = 0,\\
\hy(t) &= \hC\hx(t),
\end{aligned}
\end{equation}
where $\hx(t)\in \Rr$, $u(t)\in \R$ and $\hy(t) \in \R$ are reduced state, input and  output vectors, respectively with $r\ll n$, and all other matrices are of appropriate size. To that end, our goal is to construct reduced-order systems \eqref{eq:red_poly_sys} using Petro-Galerkin projection such that the multivariate transfer functions, as given in \eqref{eq:general_TF}, of the original system match with those of the reduced-order system at a given set of interpolation points. 
For this, we essentially require projection matrices $V\in \Rnr$ and $W \in \Rnr$, thus leading to the system matrices of \eqref{eq:red_poly_sys} as follows:
\begin{equation}\label{eq:compute_rom}
\begin{aligned}
	\hE &= W^TAV, & \hA& = W^TAV, & \hH_{\xi}& = W^TH_{\xi}V^{\circled{\tiny{$\xi$}}},\quad \xi \in \{2,\ldots,d\},\\
	\hB &= W^TB, & \hC& = CV, & \hN_{\eta}& = W^TN_{\eta}V^{\circled{\tiny{$\eta$}}},\quad \eta\in \{1,\ldots,d\}
\end{aligned}
\end{equation}
with $x(t) \approx V\hx(t)$. Clearly, the choice of the matrices $V$ and $W$ must ensure the desired interpolating properties of the original and reduced-order systems and also determines the quality of the reduced-order system. Thus, in the following theorem, we reveal the construction of the projection matrices $V$ and $W$, yielding an interpolating reduced-order system.

\begin{theorem}\label{thm:gen_interpolation}
	Consider a SISO system as given in \eqref{eq:poly_sys}. Let $\sigma_i$ and $\mu_i$, $i \in\{ 1,\ldots,\tilde r\}$, be interpolation points such that $(sE-A)$ is invertible for all $s = \{\sigma_i,\mu_i\}$, $i \in \{1,\ldots,\tilde r\}$. Moreover, let the projection matrices $V$ and $W$ be as follows:
	 \begin{subequations}\allowdisplaybreaks
		\begin{align*}
		V_{L}	&= 	\range{\Phi(\sigma_1)B,\ldots, \Phi(\sigma_{\tr})B},\\
		V_{N} 	&= 	\bigcup_{\eta= 1}^d \range{ \Phi(\lambda_{\eta+1}) N_\eta\left(\Phi(\lambda_\eta)B \otimes \cdots \otimes \Phi(\lambda_1)B \right)},  \\
		V_H 	&=  	\bigcup_{\xi= 2}^d  \range{ \Phi(\lambda_{\xi+1}) H_\xi \left(\Phi(\lambda_{\xi})B \otimes \cdots \otimes \Phi(\lambda_1)B \right) },\\ 
		W_L	&= 	\range{\Phi(\mu_1)^TC^T,\ldots, \Phi(\mu_{\tr})^TC^T},\\
		W_{N} 	&= 	\bigcup_{\eta= 1}^d \range{ \Phi(\lambda_{1})^T \left(N_\eta \right)_{(2)}\left(\Phi(\lambda_\eta)B \otimes \cdots \otimes \Phi(\lambda_2)B \otimes\Phi(\beta)^TC^T \right)}, \\
		W_H 	&=  	\bigcup_{\xi= 2}^d  \range{ \Phi(\lambda_{1})^T \left(H_\xi\right)_{(2)} \left(\Phi(\lambda_{\xi})B \otimes \cdots \otimes \Phi(\beta)^TC^T \right) },\\
		V 	&= 	\range{V_L, V_N, V_H},\\
		W	&= 	\range{W_L, W_N, W_H},
	\end{align*}
\end{subequations}
where $\Phi(s) := (sE-A)^{-1}$, $\lambda_i \in \{\sigma_1,\ldots,\sigma_{\tr}\}, i \in \{1,\ldots,d+1\}, \beta \in \{\mu_1,\ldots,\mu_{\tr}\}$, and $\left(H_\xi\right)_{(2)} \in \R^{n\times n^\xi}$ and $\left(N_\eta\right)_{(2)} \in \R^{n\times m\cdot n^\xi}$ are, respectively, the mode-2 matricizations of the $(\xi{+}1)$-way tensor $\bten{H}_\xi \in \R^{n\times \cdots \times n}$ and $(\eta{+}2)$-way tensor $\bten{N}_\eta \in \R^{n\times \cdots \times n}$ whose mode-1 matricizations are $H_\xi$  and $N_\eta$, respectively.  Assume $V$ and $W$ are of full column rank. If a reduced-order system is computed as shown in \eqref{eq:compute_rom} using the matrices $V$ and $W$, then the reduced-order system satisfies the following interpolation conditions:
	\begin{subequations}
	\begin{align}
	\bF_L(\lambda_1)					&= 	\hat{\bF}_L(\lambda_1), \label{eq:linear1} \\
	\bF_L(\beta)						&= 	\hat{\bF}_L(\beta), \label{eq:linear2}\\	
	\bF_N^{(\eta)}(\lambda_1,\ldots,\lambda_{\eta+1}) 	&= 	\hat \bF_N^{(\eta)}(\lambda_1,\ldots,\lambda_{\eta+1}),\label{eq:bilinear1} \\
	\bF_N^{(\eta)}(\lambda_1,\ldots,\lambda_{\eta},\beta) 	&=	\hat \bF_N^{(\eta)}(\lambda_1,\ldots,\lambda_{\eta},\beta), \label{eq:bilinear2}\\
	\bF_H^{(\xi)}(\lambda_1,\ldots,\lambda_{\xi+1}) 	&= 	\hat \bF_H^{(\xi)}(\lambda_1,\ldots,\lambda_{\xi+1}), \label{eq:Qbilinear1}\\
	\bF_H^{(\xi)}(\lambda_1,\ldots,\lambda_{\eta},\beta) 	&= 	\hat \bF_H^{(\xi)}(\lambda_1,\ldots,\lambda_{\eta},\beta)\label{eq:Qbilinear2}.
	\end{align}
\end{subequations}
\end{theorem}
\begin{proof}
The relations, given in \cref{eq:linear1} and \cref{eq:linear2} follows directly from the linear case, see, e.g., \cite{morAnt05}. Therefore, we omit their proofs for the sake of brevity of the paper. However, for the rest of the proof,  we note down intermediate results, which can be obtained while proving \cref{eq:linear1} and~\cref{eq:linear2}:
\begin{subequations}\label{eq:linearRel}
 \begin{align}
  V\hat\Phi(\lambda_1)\hB 	&= \Phi(\lambda_1)B,  &\lambda_1 &\in \{\sigma_1,\ldots,\sigma_{\tilde r}\} \label{eq:linearRel_V}\\
  W\Phi(\beta)^T\hC 		&= \Phi(\beta)^TC^T,  &\beta &\in \{\mu_1,\ldots,\mu_{\tilde r}\},\label{eq:linearRel_W}
 \end{align}
\end{subequations}
where $\Phi(s) = (sE-A)^{-1}$ and $\hPhi(s) = (s\hE-\hA)^{-1}\hB$. Now, we focus on the relation \cref{eq:bilinear1}. We begin with
\begin{align}
    &V\hPhi(\lambda_{\eta+1})\hN_\eta\left(\hPhi(\lambda_\eta)\hB \otimes \cdots \otimes \hPhi(\lambda_1)\hB \right) \nonumber\\
    &\qquad\qquad = V\hPhi(\lambda_{\eta+1})W^TN_\eta V^{\circled{\tiny{$\eta$}}}\left(\hPhi(\lambda_\eta)\hB \otimes \cdots \otimes \hPhi(\lambda_1)\hB \right)\nonumber\\
    &\pushright{\left(\because \hN_\eta = W^TN_\eta \kronF{V}{$\eta$}\right)}\nonumber\displaybreak[3]\\
    &\qquad\qquad = V\hPhi(\lambda_{\eta+1})W^TN_\eta \left( V \hPhi(\lambda_\eta)\hB \otimes \cdots \otimes V\hPhi(\lambda_1)\hB \right)\nonumber\\
    &\qquad\qquad = V\hPhi(\lambda_{\eta+1})W^TN_\eta \left(  \Phi(\lambda_\eta)B \otimes \cdots \otimes \Phi(\lambda_1)B \right)\nonumber\\
    &\pushright{\left(\text{using}~ \cref{eq:linearRel_V}\right)}\nonumber\\
    &\qquad\qquad = V\hPhi(\lambda_{\eta+1})W^T \Phi(\lambda_{\eta+1})^{-1} \underbrace{\Phi(\lambda_{\eta+1}) N_\eta \left(  \Phi(\lambda_\eta)B \otimes \cdots \otimes \Phi(\lambda_1)B \right)}_{\in V}\nonumber\\
    &\pushright{\left(\text{introduction of } I_n = \Phi(\lambda_{\eta+1})^{-1} \Phi(\lambda_{\eta+1}) \right)}\nonumber\\
    &\qquad\qquad = V\hPhi(\lambda_{\eta+1})W^T \Phi(\lambda_{\eta+1})^{-1} Vz,\label{eq:VPhi}
\end{align}
where the vector $z$ is such that $Vz = \Phi(\lambda_{\eta+1}) N_\eta \left(  \Phi(\lambda_\eta)B \otimes \cdots \otimes \Phi(\lambda_1)B \right)$. Additionally, we have
\begin{align*}
\hPhi(s) W^T\Phi(s)^{-1}V &=    (s\hE-\hA)^{-1} W^T(sE-A)V \\ &= (s\hE-\hA)^{-1} (sW^TEV-W^TAV) = I_r.
\end{align*}
Substituting the above relation in \cref{eq:VPhi} and pre-multiplying with  $C$ yields the relation \cref{eq:bilinear1}. Similarly, we can prove the relation \cref{eq:Qbilinear1}. Next, we focus on the relation \cref{eq:bilinear2}. We know that
\begin{equation*}
 \hN_\eta= W^T N_\eta\left. \kronF{V}{$\eta$}\right..
\end{equation*} 
Hence, using \cref{eq:matricizationRelationMatrix}, we obtain
\begin{equation}\label{eq:NtensorRel}
 \left(\hN_\eta\right)_{(2)} = V^T\left(N_\eta\right)_{(2)}\left( \kronF{V}{$\eta{-}1$}\otimes W\right),
\end{equation}
where $\left(\hN_\eta\right)_{(2)}$ is the mode-2 matricization of the tensor $\bten{\hN}_\eta$ whose mode-1 matricization is $\hN_\eta$. With the relation \cref{eq:NtensorRel}, we now consider
\begin{subequations}\allowdisplaybreaks
\begin{align*}
 &W\hPhi(\lambda_{1})^T \left(\hN_\eta \right)_{(2)}\left(\hPhi(\lambda_\eta)\hB \otimes \cdots \otimes \Phi(\lambda_2)\hB \otimes\hPhi(\beta)^T\hC^T  \right)\\
 &\quad =W\hPhi(\lambda_{1})^T V^T\left(N_\eta\right)_{(2)}\left( \kronF{V}{$\eta{-}1$}\otimes W\right)\Big(\hPhi(\lambda_\eta)\hB \otimes \cdots \\ 
 & \pushright{\otimes \Phi(\lambda_2)\hB \otimes\hPhi(\beta)^T\hC^T  \Big) }\\
 & \pushright{\left(\text{using}~\cref{eq:NtensorRel}\right)}\\
 & \quad= W\hPhi(\lambda_{1})^T V^T\left(N_\eta\right)_{(2)}\left( V\hPhi(\lambda_\eta)\hB \otimes \cdots \otimes V\hPhi(\lambda_2)\hB \otimes W \hPhi(\beta)^T\hC^T \right) \\ 
 & \quad= W\hPhi(\lambda_{1})^T V^T\left(N_\eta\right)_{(2)}\left(s \Phi(\lambda_\eta)B \otimes \cdots \otimes \Phi(\lambda_2)B \otimes  \Phi(\beta)^TC^T \right) \\
 & \pushright{\left(\text{using}~\cref{eq:linearRel}\right)}\\
 & \quad= W\hPhi(\lambda_{1})^T V^T\Phi(\lambda_{1})^{-T} \\
 &\qquad\qquad\qquad \times \underbrace{\Phi(\lambda_{1})^T \left(N_\eta\right)_{(2)}\left( \Phi(\lambda_\eta)B \otimes \cdots \otimes \Phi(\lambda_2)B \otimes  \Phi(\beta)^TC^T \right)}_{\in W \left(\therefore ~=: Wq\right)} \\
 &\quad= W\hPhi(\lambda_{1})^T V^T\Phi(\lambda_{1})^{-T} Wq  = Wq\\
 &\quad =\Phi(\lambda_{1})^T \left(N_\eta\right)_{(2)}\left( \Phi(\lambda_\eta)B \otimes \cdots \otimes \Phi(\lambda_2)B \otimes  \Phi(\beta)^TC^T \right).
\end{align*}
\end{subequations}
Next, we multiply both sides by $B^T$ to get
\begin{multline*}
 \hB^T\hPhi(\lambda_{1})^T \left(\hN_\eta \right)_{(2)}\left(\hPhi(\lambda_\eta)\hB \otimes \cdots \otimes \Phi(\lambda_2)\hB \otimes\hPhi(\beta)^T\hC^T  \right) \\
   = B\Phi(\lambda_{1})^T \left(N_\eta\right)_{(2)}\left( \Phi(\lambda_\eta)B \otimes \cdots \otimes \Phi(\lambda_2)B \otimes  \Phi(\beta)^TC^T \right).
\end{multline*}
Using the matricization property of tensor-vector multiplications~\cref{eq:matricizationRelation}, we get 
\begin{equation*}
 \hC\hPhi(\beta) \hN_\eta\left(\hPhi(\lambda_\eta)\hB \otimes \cdots \otimes \hPhi(\lambda_1)\hB \right) 
  = C\Phi(\beta) N_\eta\left(s\Phi(\lambda_\eta)B \otimes \cdots \otimes \Phi(\lambda_1)B \right),
\end{equation*}
which is nothing but the relation given in \cref{eq:bilinear2}. Using similar steps, we can prove \cref{eq:Qbilinear2}; thus, for the sake of brevity, we skip it. This concludes the proof. 
\end{proof}
\subsection{Tangential-interpolating ROMs for MIMO systems}
In this subsection, we discuss a construction of an interpolating reduced-order systems for MIMO polynomial systems. Similar to the SISO case, the leading generalized transfer functions for a MIMO polynomial system are given
as follows:
\begin{subequations}\label{eq:general_TF_MIMO}
\begin{align}
 \bF_L(s_1) & =  C\Phi(s_1)B,\\
 \bF^{(\xi)}_H(s_1,\ldots,s_{\xi+1}) &=  C\Phi(s_{\xi+1}) H_\xi \left(\Phi(s_\xi)B \otimes \cdots \otimes \Phi(s_{1})B\right),\\
 \bF_N^{(\eta)}(s_1,\ldots, s_{\eta+1}) & =  C \Phi(s_{\eta+1}) N_\eta\left(I_m\otimes\Phi(s_{\eta})B \otimes \cdots \otimes \Phi(s_{1})B \right),
\end{align}
\end{subequations}
where $\Phi(s) = (sI_n{-}A)^{-1}$. In \Cref{thm:gen_interpolation}, we have provided a general interpolation framework for SISO polynomial systems, which can be extended to the MIMO case. However, a straightforward extension of the interpolation idea for the MIMO case might lead to projection matrices $V$ and $W$ with an unmanageable number of columns for the MIMO case.  Therefore, we make use of the tangential interpolation concept from the linear case for MIMO systems \cite{morGalVV04}. Furthermore, for ease of practical implementation, we avoid vectors in the projection matrices $V$ and $W$ corresponding to  cross frequencies. This means that we set $\lambda_1 = \lambda_2 = \cdots = \lambda_{\{\eta,\xi\}} = \beta$.    As a result, we propose the following lemma that is arguably of more importance from a practical point of view. 
\begin{lemma}\label{lemma:singleInterpolation}
	Consider the original system as given in \eqref{eq:poly_sys}. Let $\sigma_i \in \C$, $i \in \{1,\ldots,\tilde r\}$, be interpolation points such that $sE-A$ is invertible for all $s \in \{\sigma_1,\ldots, \sigma_{\tilde r}\}$, and $b_i \in \C^m$ and $c_i\in\C^q$ $i \in \{1,\ldots,{\tilde r}\}$ be right and left tangential directions corresponding to $\sigma_i$, respectively. Let $V$ and $W$ be defined as follows: 
\begin{subequations}\allowdisplaybreaks
	\begin{align*}
	V_{L}	&= 	\bigcup_{i= 1}^{\tr} \range{\Phi(\sigma_i)Bb_{i}},\\
	V_{N} 	&= 	\bigcup_{\eta= 1}^d\bigcup_{i= 1}^{\tr} \range{ \Phi(\sigma_i) N_\eta\left(I_m\otimes \Phi(\sigma_i)Bb_i \otimes \cdots \otimes \Phi(\sigma_i)Bb_i \right)},  \\
	V_H 	&=  	\bigcup_{\xi= 2}^d\bigcup_{i= 1}^{\tr}  \range{ \Phi(\sigma_i) H_\xi \left(\Phi(\sigma_i)Bb_i \otimes \cdots \otimes \Phi(\sigma_i)Bb_i \right) },\\ 
	W_L	&= 	\bigcup_{i= 1}^{\tr}\range{\Phi(\sigma_i)^TC^Tc_i},\\
	W_{N} 	&= 	\bigcup_{\eta= 1}^d \bigcup_{i= 1}^{\tr}\range{ \Phi(\sigma_{i}) \left(N_\eta \right)_{(2)}\left(I_m\otimes\Phi(\sigma_i)B \otimes \cdots \otimes \Phi(\sigma_i)B \otimes\Phi(\sigma_i)^TC^T \right)}, \\
	W_H 	&=  	\bigcup_{\xi= 2}^d \bigcup_{i= 1}^{\tr} \range{ \Phi(\sigma_{i}) \left(H_\xi\right)_{(2)} \left(\Phi(\sigma_{i})B \otimes \cdots \otimes \Phi(\sigma_{i})B\otimes \Phi(\beta)^TC^T \right) },\\
	V 	&= 	\range{V_L, V_N, V_H},\\
	W	&= 	\range{W_L, W_N, W_H}.
	\end{align*}
\end{subequations}
If a reduced-order system is computed as shown in \eqref{eq:compute_rom} using the projection matrices $V$ and $W$, where we assume $V$ and $W$ to be of full rank, then the following interpolation conditions are fulfilled:
	\begin{subequations}\allowdisplaybreaks
		\begin{align}
		\bF_L(\sigma_i) b_i	&= 	\hat{\bF}_L(\sigma_i)b_i, \label{eq:1}\\
		c_i^T\bF_L(\sigma_i)	&= 	c_i^T\hat{\bF}_L(\sigma_i),\label{eq:2} \\
		\dfrac{d}{ds_1} c_i^T\bF_L(\sigma_i)b_i
					&= 	\dfrac{d}{ds_1} c_i^T\hat{\bF}_L(\sigma_i)b_i,  \label{eq:3}\\
		\bF_N^{(\eta)}(\sigma_i,\ldots,\sigma_i)\left(I_m\otimes b_i^{\circled{\tiny {$\eta$}}}\right) 
					&= 	\hat F_N^{(\eta)}(\sigma_i,\ldots,\sigma_i) \left(I_m\otimes b_i^{\circled{\tiny {$\eta$}}}\right), \label{eq:4}\\
		c_i^T\bF_N^{(\eta)}(\sigma_i,\ldots,\sigma_i)\left( I_m^{\circled{\tiny {$2$}}} \otimes  b_i^{\circled{\tiny {$\eta{-}1$}}}\right) 
					&= 	c_i^T\hat{\bF}_N^{(\eta)}(\sigma_i,\ldots,\sigma_i)\left( I_m^{\circled{\tiny {$2$}}} \otimes  b_i^{\circled{\tiny {$\eta{-}1$}}}\right)\label{eq:5}\\
		\dfrac{d}{ds_j}c_i^T\bF_N^{(\eta)}(\sigma_i,\ldots,\sigma_i)\left(I_m\otimes b_i^{\circled{\tiny {$\eta$}}}\right) 
					&= 	\dfrac{d}{ds_j}  c_i^T \hat F_N^{(\eta)}(\sigma_i,\ldots,\sigma_i) \left(I_m\otimes b_i^{\circled{\tiny {$\eta$}}}\right), \label{eq:6}\\
		\bF_H^{(\xi)}(\sigma_i,\ldots,\sigma_i) b_i^{\circled{\tiny {$\xi$}}} 
					&= 	\hat{\bF}_H^{(\xi)}(\sigma_i,\ldots,\sigma_i)  b_i^{\circled{\tiny {$\xi$}}}, \label{eq:7}\\
		c_i^T\bF_H^{(\xi)}(\sigma_i,\ldots,\sigma_i)\left( I_m \otimes  b_i^{\circled{\tiny {$\xi{-}1$}}}\right) 
					&= 	c_i^T\hat{\bF}_H^{(\xi)}(\sigma_i,\ldots,\sigma_i)\left( I_m \otimes  b_i^{\circled{\tiny {$\xi{-}1$}}}\right),\label{eq:8}\\
		\dfrac{d}{ds_j}c_i^T\bF_H^{(\xi)}(\sigma_i,\ldots,\sigma_i) b_i^{\circled{\tiny {$\xi$}}} 
					&= 	\dfrac{d}{ds_j}  c_i^T\hat{\bF}_H^{(\xi)}(\sigma_i,\ldots,\sigma_i) b_i^{\circled{\tiny {$\xi$}}}\label{eq:9}
		\end{align}
	\end{subequations}
	where $i \in \{1,\ldots, \tilde r\}$, $\xi \in \{2,\ldots,d\} $, $\eta \in \{1,\ldots,d\}$ and $\dfrac{d}{ds_j}$ denotes the partial derivative with respect to $s_j$ of a given function. 
\end{lemma}
\begin{proof}
 The proof of \cref{eq:1,eq:2,eq:4,eq:5,eq:7,eq:8} exactly follows the proof of \cref{thm:gen_interpolation}. Using very similar steps and simple algebra, one can easily prove the rest of the conditions. 
\end{proof}
\subsection{Connection to the Loewner Approach}
In recent years, Loewner-based MOR has received a lot of attention. For linear systems, the authors in \cite{mayo2007framework} have discussed the Loewner approach to construct reduced-order systems using transfer function data.  Later on, the Loewner approach has been extended to other classes of nonlinear systems, namely bilinear and QB systems in \cite{antoulas2016model,gosea2018data}, where data related to generalized transfer functions is required to obtain a reduced-order system. 

An important ingredient in the Loewner approach is the construction of the Loewner matrix ($\mathbb L$) and the shifted Loewner matrix ($\mathbb L_s$). One way to construct the matrices $\LL$ and $\LL_s$ is either by using an experimental set-up or by using numerical evaluations of the generalized transfer functions, which is the primary inspiration of the method. However, there is a strong  connection with interpolation of (generalized-) transfer functions, corresponding to a given system. As a result, we, alternatively, can  construct the latter matrices by projection for a given realization of a system, ensuring the interpolation conditions. 

For an example, let us consider $4$ frequency measurements $H(\sigma_1), H(\sigma_2)$, $H(\mu_1)$ and $H(\mu_2)$, where $H(s) :=  C(sE-A)^{-1}B \in \C$ is the transfer function of a linear SISO system with the system matrices $(E,A,B,C)$. As shown e.g., in \cite{antoulas2014tutorial}, the matrices $\LL$ and $\LL_s$, using the data points and letting $\sigma_{\{1,2\}}$ and $\mu_{\{1,2\}}$ to be the right and left interpolation points, can be constructed as follows:
\begin{equation}\label{eq:loewnerLinear}
  \LL{(i,j)} = \dfrac{H(\mu_i) - H(\sigma_j)}{\mu_i - \sigma_j},\qquad  \LL_s{(i,j)} = \dfrac{\mu_iH(\mu_i) - \sigma_jH(\sigma_j)}{\mu_i - \sigma_j}, 
\end{equation}
where ${i,j} \in \{1,2\}$. Moreover, if the matrices $V$ and $W$ are given as in \Cref{thm:gen_interpolation}, i.e.,
\begin{align*}
 V &= \begin{bmatrix} (\sigma_1 E-A)^{-1}B,& (\sigma_2 E-A)^{-1}B\end{bmatrix},\\
 W &= \begin{bmatrix} (\mu_1 E-A)^{-T}C^T,& (\mu_2 E-A)^{-T}C^T\end{bmatrix},
\end{align*}
then the matrices $\LL$ and $\LL_s$, shown in \cref{eq:loewnerLinear}, can also be constructed as 
\begin{equation}\label{eq:loewnerLinearProj}
 \LL = -W^TEV, \quad \LL_s = -W^TAV.
\end{equation}

A similar analogy can also be seen for bilinear and QB systems \cite{antoulas2016model,gosea2018data}. It is preferable to construct $\LL$ and $\LL_s$ using the data if the data corresponding to the transfer function can be either computed cheaply by its explicit expression or determined by an experimental setup. However, in the case nonlinear systems, it is not straightforward to determine the generalized transfer function by an experiment, which is mainly due to not having a clear interpretation of generalized transfer functions of nonlinear systems as in the case of linear systems. Thus, the method to determine $\LL$ and $\LL_s$ by projection shown in \cref{eq:loewnerLinearProj}, can be of greater use when measurement data is not available but instead, we have a system realization. 

In this paper, we assume that a realization of the polynomial systems \eqref{eq:poly_sys} is given and thus focus on constructing the matrices  $\LL$ and $\LL_s$ using projection \eqref{eq:loewnerLinearProj}, and the rest of the system matrices using the same projection matrices $V$ and $W$ are given as follows:
\begin{align*}
 \BB &= W^TB, &\HH_\xi &= W^TH_\xi V^{\circled{\tiny{$\xi$}}}, && \xi \in \{2,\ldots,d\},\\
 \CC &= CV, & \NN_\eta &= W^TN_\eta \left(I_m\otimes V^{\circled{\tiny{$\eta$}}}\right), &&\eta \in \{1,\ldots,d\}.
\end{align*}

By \Cref{thm:gen_interpolation}, it is clear that the systems $\cS_1:\left(E,A, H_\xi, N_\eta, B,C\right)$ and $\cS_2:\left(\LL,\LL_s, \HH_\xi, \NN_\eta, \BB,\CC\right)$ are interpolating at the considered frequency points. However, $\cS_2$ can be singular, meaning that it may contain a lot of redundant information which can be compressed. Thus, inspired by the Loewner approach for linear, bilinear, and QB systems, we  remove the redundancy by compressing the information using an SVD of the following matrices, composed of $\LL$ and $\LL_s$:
\begin{align}
 \bbm \LL & \LL_s \ebm &= Y_1\Sigma_1 X^T_1,\\
 \bbm \LL~\\ \LL_s \ebm &= Y_2\Sigma_2 X^T_2,
\end{align}
where the diagonal entries of $\Sigma_1$ and $\Sigma_2$ are in non-increasing order. Based on the first $r$ columns of $Y_1$ and $X_2$, denoted by $Y_r$ and $X_r$, we can determine a compressed $\hat{\cS}_2$, compressing the information of $\cS_2$, as follows:
\begin{align*}
\hE &= Y_r^T \LL X_r,  & \hA &= Y_r^T\LL_s X_r^T,& \hH_\xi &= Y_r^T \HH_\xi X_r^{\circled{\tiny{$\xi$}}}, \qquad\qquad\quad \xi \in \{2,\ldots,d\},\\
  \hB &= Y_r^T \BB, & \hC &= \CC X_r,  &\hN_\eta &= Y_r^T\NN_\eta \left(I_m\otimes X_r^{\circled{\tiny{$\eta$}}}\right),\quad \eta \in \{1,\ldots,d\}.
\end{align*}

There are essentially two steps involved in order to get $\hat{\cS}_2$. In the first step, we require matrices such as $\HH_\xi$ and $\NN_\eta$, which are generally dense, hence unmanageable. This is followed by compressing these matrices by using $X_r$ and $Y_r$. However, upon closer inspection, we can determine $\hat{\cS}_2$ without completely forming $\cS_2$, or matrices $\HH_\xi$ and $\NN_\eta$, but we can rather determine $\hat{\cS}_2$ by directly projecting the original system matrices using appropriate projection matrices. If we define the \emph{effective} projection matrices as follows:
\begin{equation}
 \Ve := VX_r, \quad\text{and} \quad\We := WY_r,
\end{equation}
then $\hat{\cS}_2$ can be determined in a traditional projection framework of the original system~\eqref{eq:poly_sys} as follows:
\begin{equation}\label{eq:deter_red_direct}
\begin{aligned}
\hE &= \We^TE\Ve,  & \hA &= \We^TA\Ve,& \hH_\xi &= \We H_\xi\Ve^{\circled{\tiny{$\xi$}}},\\
  \hB &= \We^T B, & \hC &= C \Ve,  &\hN_\eta &= \Ve^T N_\eta \left(I_m\otimes\Ve^{\circled{\tiny{$\eta$}}}\right),\end{aligned}
  \end{equation}
where $\xi\in \{2,\ldots,d\}$ and $\eta \in \{1,\ldots,d\}$. 
We point out that it is advantageous to determine  reduced system matrices, or the system $\hat{\cS}_2$ as shown in \eqref{eq:deter_red_direct}; this way, we are not required to form large dense matrices such as $\HH_\xi$ and $\NN_\eta$. We can rather compute reduced matrices by  multiplying efficiently the sparse and super-sparse\footnote{super-sparsity of a matrix is defined as a ratio of the number of non-zero distinct numbers  to the total number of non-zero elements.} original matrices $H_\xi$ and $N_\eta$ with $\Ve$ and $\We$. Having all these results, we briefly sketch the steps to determine reduced-order systems in \Cref{algo:nonParaPoly}. However, an important computational aspect related to tensor computations such as $\We H_\xi\Ve^{\circled{\tiny{$\xi$}}}$ still remains, which is discussed in the next section.

\begin{algorithm}[!tb]
	\caption{MOR for Non-Parametric Polynomial Systems (\texttt{LbNPS-Algo}).}
	\label{algo:nonParaPoly}
	\begin{algorithmic}[1]
		\Statex {\bf Input:} The system matrices $E,A,H_\xi,N_\eta,B,C$, $\xi \in \{2,\ldots,d\}$, $\eta \in \{1,\ldots,d\}$ and a set of interpolation points $\sigma_i$ and corresponding tangential directions $b_i$ and $c_i$, the reduced order $r$.
		\Statex {\bf Output:}  The reduced system matrices $\hE,\hA,\hH_\xi,\hN_\eta,\hB,\hC$, $\xi \in \{2,\ldots,d\}$, $\eta \in \{1,\ldots,d\}$.
		\State Determine $V$ and $W$ as shown in \Cref{lemma:singleInterpolation}.
		\State Define Loewner and shifted Loewner matrices as follows:
		$$
		\LL = -W^TE V,\qquad 
		\LL_s = -W^TA V, 
		$$
		\State Compute SVD of the matrices:
		\begin{align*}
		\bbm \LL,  \LL_s \ebm &= Y_1\Sigma_1 X^T_1,\\
		\bbm \LL~ \\  \LL_s \ebm &= Y_2\Sigma_2 X^T_2.
		\end{align*}
		\State Define $Y_r := Y_1(:,1\mathbin{:}r)$ and $X_r := X_2(:,1\mathbin{:}r)$. 
		\State Determine compact projection matrices:
		\Statex\quad $\Ve := \texttt{orth}\left(VX_r\right) $ and $\We := \texttt{orth}\left(WY_r\right)$.
		\State Determine the reduced-order system as follows:
		\Statex $\begin{aligned}
\hE &= \We^TE\Ve,  & \hA &= \We^TA\Ve,& \hH_\xi &= \We H_\xi\Ve^{\circled{\tiny{$\xi$}}},\quad &&\xi\in \{2,\ldots,d\},\\
  \hB &= \We^T B, & \hC &= C \Ve,  &\hN_\eta &= \Ve^T N_\eta \left(I_m\otimes\Ve^{\circled{\tiny{$\eta$}}}\right), \quad &&\eta \in \{1,\ldots,d\}.
\end{aligned}$
	\end{algorithmic}
\end{algorithm}

\section{Computational Aspects and Application of CUR}\label{sec:computationalCUR}
In this section, we discuss two important computational aspects which are related to evaluating the nonlinear terms of the reduced-order systems \eqref{eq:deter_red_direct} and the use of the CUR matrix approximation in order to accelerate simulations of the reduced-order systems.
\subsection{Efficient evaluation the nonlinear terms of the ROMs}
Let us begin with the computational effort related to evaluating, e.g., $\hat H_\xi := \We^T H_\xi\Ve^{\circled{\tiny{$\xi$}}}$. It can be noticed that a direct computation of the above terms requires the computation of $\Ve^{\circled{\tiny{$\xi$}}}$. Generally, the matrix $\Ve$ is a dense matrix; thus, the computation related to $\Ve^{\circled{\tiny{$\xi$}}}$ is of complexity $\mathcal O((n{\cdot} r)^\xi)$, which easily becomes an unmanageable task. For $\xi = 2$, the authors in \cite{morBenB15} have proposed a method using tensor algebra to compute $\hat H_2$ without explicitly forming $\Ve\otimes \Ve$. On the other hand, the authors in \cite{morBenGG18} have aimed at exploiting the structure of the nonlinear operators, typically arising in PDEs/ODEs, thus also leading to an efficient method to compute $\hat H_2$. 

In this paper, we focus on the latter approach, where the explicit nonlinear operator of the PDEs is utilized, to compute $\hH_\xi$. Extending the discussion in \cite{morBenGG18}, in principle, we can write the term $H_\xi x^{\circled{\tiny{$\xi$}}}$ in the system \eqref{eq:poly_sys} in the Hadamard product form as follows:
\begin{equation}\label{eq:nonlinearTerm}
 H_\xi  x^{\circled{\tiny{$\xi$}}} = \cA_1 x \circ \cdots \circ \cA_\xi x,
\end{equation}
where $\circ$ denotes the Hadamard product and $\cA_i \in \Rnn$ are the constant matrices depending on the nonlinear operator in a governing equation. In order to reduce these nonlinear terms, resulting in a reduced-order system, we proceed as follows. Firstly, we substitute $x(t) \approx \Ve \hat x(t)$, where $x(t) \in \Rn$ and $\hx(t) \in \Rr$ are the original and reduced  state vectors, respectively, and then multiply $\We^T$ from the left-hand side, thus leading to the corresponding  nonlinear term:
\begin{equation*}
 \hH_\xi  \hx^{\circled{\tiny{$\xi$}}} =  \We^T\left( \left(\hat{\cA}_{1} \hx\right) \circ \cdots \circ \left(\hat{\cA}_{\xi} \hx\right)\right),
\end{equation*}
where $\hat{\cA}_{i} = \cA_i\Ve$, $i \in \{1,\ldots,\xi\}$. Next, we use the relation between a Hadamard product and the Kronecker product, that is 
\begin{equation*}
 \cP p \circ \cQ q = \begin{bmatrix}  \cP(1,:)\otimes \cQ(1,:) \\ \vdots \\  \cP(n,:)\otimes \cQ(n,:) \end{bmatrix} (p\otimes q).
\end{equation*}
Thus, we get 
\begin{align}\label{eq:reduced_H}
 \We^T\left(\left(\hat{\cA}_{1} \hx\right) \circ \cdots \circ \left(\hat{\cA}_{\xi} \hx\right)\right) = \We^T\underbrace{\begin{bmatrix}\hat{\cA}_{1}(1,:)  \otimes \cdots \otimes \hat{\cA}_{\xi}(1,:) \\ \vdots \\ \hat{\cA}_{1}(n,:)  \otimes \cdots \otimes \hat{\cA}_{\xi}(n,:)\end{bmatrix}}_{=:\hat{\cA}}\hx^{\circled{\tiny{$\xi$}}}.
\end{align}
It can be seen that $\We^T\hat{\cA} =  \hat H_\xi$. Summarizing, we can perform computations related to $ \hat H_\xi$ efficiently by utilizing the particular structure of the nonlinear terms in PDEs/ODEs, without  explicitly  forming $\Ve^{\circled{\tiny{$\xi$}}}$. We illustrate the procedure for a typical nonlinear PDE term in \Cref{subsec:ComputationalIllustration}.

\subsection{CUR matrix approximation and ROMs}
Next, we discuss another computational issue, due to which we may not achieve the desired reduction in the simulation time even after reducing the original system~\eqref{eq:poly_sys}. Explaining  this issue further, the reduced matrices such as $\hH_\xi \in \R^{r\times r^{\xi}}$ are generally dense matrices which are multiplied with $\hx^{\circled{\tiny{$\xi$}}}$. Thus, the computation $\hH_\xi \hx^{\circled{\tiny{$\xi$}}}$ is in $\cO(r^{2\xi+1})$, which increases rapidly with the order of the reduced system or polynomial system \eqref{eq:poly_sys}. As a remedy, in this paper, we propose a new procedure  to approximate $\hH_\xi \hx^{\circled{\tiny{$\xi$}}}$, which can be computed cheaply. For this, we make use of the \emph{CUR matrix approximation}, see, e.g.~\cite{mahoney2009cur,sorensen2016deim,wang2013improving}. Using this, we can approximate  the matrix $\hat{\cA}$, defined in \eqref{eq:reduced_H}, as follows: 
\begin{equation}\label{eq:hA_CUR}
 \hat{\cA} \approx \cC \cU \cR,
\end{equation}
where $\cC \in \R^{n \times n_c}$ and $\cR \in \R^{n_r\times r^l}$ contain wisely chosen $n_c$ columns and $n_r$ rows of the matrix $\hat{\cA}$, respectively, and $\cU \in \R^{n_c\times n_r}$ is determined such that it minimizes $\|\hat{\cA} - \cC\cU\cR\|$ in an appropriate norm. There has been a significant research how to choose  columns and rows appropriately, leading to a good or even optimal in some sense, approximation of a matrix. We refer the reader to \cite{mahoney2009cur,sorensen2016deim,wang2013improving} and references therein for more details. Substituting the relation \cref{eq:hA_CUR}  in \cref{eq:reduced_H} results in
\begin{align}
 \We^T\hat{\cA} \hx^{\circled{\tiny{$\xi$}}}  \approx \We^T \cC \cU \cR\hx^{\circled{\tiny{$\xi$}}}.
\end{align}
Next, we closely look at the term $\cR\hx^{\circled{\tiny{$\xi$}}}$, whose columns are given as 
\begin{equation}
 \left(\hat{\cA}_{1}(i_r,:)  \otimes \cdots \otimes \hat{\cA}_{\xi}(i_r,:) \right)\hx^{\circled{\tiny{$\xi$}}},
\end{equation}
where $i_r$ belongs to the columns chosen by the CUR matrix approximation. We know that $\hat{\cA}_{1}(i_r,:) = {\cA}_{1}(i_r,:)V$. Substituting this relation and $x \approx V\hat x$, we get 
\begin{align*}
 &\hat{\cA}_{1}(i_r,:)  \otimes \cdots \otimes \hat{\cA}_{\xi}(i_r,:) \hx^{\circled{\tiny{$\xi$}}} \\
 &\hspace{3cm} = \left({\cA}_{1}(i_r,:)  \otimes \cdots \otimes {\cA}_{\xi}(i_r,:)\right)V^{\circled{\tiny{$\xi$}}}  \hx^{\circled{\tiny{$\xi$}}} \\
 &\hspace{3cm}\approx \left({\cA}_{1}(i_r,:)  \otimes \cdots \otimes {\cA}_{\xi}(i_r,:)\right) x^{\circled{\tiny{$\xi$}}} := \texttt{NL}_{i_r}.
\end{align*}
Comparing the above quantity with \eqref{eq:nonlinearTerm}, it can be noticed that the quantity $\texttt{NL}_{i_r}$ is nothing but the computation of the corresponding nonlinearity of the original system at a particular grid point. Furthermore, the term $\We^T\cC\cU \in \R^{r\times n_r}$ can be precomputed. This idea is very closely related to empirical interpolation methods, which are commonly used in reduced basis methods or proper orthogonal decomposition for nonlinear systems to reduce the computational cost related to nonlinear terms \cite{morBarMNetal04,morChaS10}. 
\subsection{An illustration using Chafee-Infante equation}\label{subsec:ComputationalIllustration}
 In the following, we illustrate the computation of the reduced nonlinear term $\hat H_\xi$ and the usage of the CUR decomposition with the help of the Chafee-Infante equation. At this stage, we avoid describing the governing PDEs of the Chafee-Infante equation; we provide a detailed description of it in the numerical section. However, at the moment, we just note that it has cubic nonlinearity, i.e., $-v^3$, where $v$ is the dependent variable. Hence, if the system is written in the form given in \cref{eq:poly_sys}, we have the following nonlinear term:
 \begin{equation*}
	H_3 x^{\circled{\tiny{$3$}}} := -x\circ x \circ x.
 \end{equation*}
If the above term is reduced using the projection matrices $\Ve$ and $\We$ as shown in \cref{eq:deter_red_direct}, then we obtain
\begin{align}
 \We^TH_3 \Ve^{\circled{\tiny{$3$}}}\hx^{\circled{\tiny{$3$}}} &=  \hH_3 \hx^{\circled{\tiny{$3$}}}\nonumber\\
 &=  \We^T\left(\Ve \hx\circ \Ve \hx \circ\Ve \hx\right) \nonumber\\
 &=  \We^T \underbrace{\begin{bmatrix} \Ve(1,:) \otimes \Ve(1,:)\otimes \Ve(1,:) \\ \vdots \\  \Ve(n,:)\otimes \Ve(n,:)\otimes \Ve(n,:) \end{bmatrix}}_{=:\cVe} \left(\hx\otimes \hx\otimes \hx\right). \label{eq:H3_computation}
\end{align}
\Cref{eq:H3_computation} shows that instead of explicitly forming $\Ve^{\circled{\tiny{$3$}}}$ to determine $\hH_3$, we can rather compute it  by a smart choice of rows and perform the Kronecker products as shown in \eqref{eq:H3_computation}.    Furthermore, as discussed earlier, the evaluation of the term $\hH_3 \hx^{\circled{\tiny{$3$}}}$, in general, is of complexity $\cO(r^7)$, which might be expensive if the order of the reduced system  $(r)$ is notable. Also, we stress that the term  $\hH_3 \hx^{\circled{\tiny{$3$}}}$ needs to be computed at each time step for every simulation.  To ease that we aim at further approximating  $\hH_3 \hx^{\circled{\tiny{$3$}}}$.  Thus, we first apply the CUR matrix approximation to the matrix $\cVe$, defined in \eqref{eq:H3_computation}, to approximate it by using selected columns and rows, that is 
\begin{equation}\label{eq:CUR_Veff}
 \cVe \approx \cC_{\ensuremath{\mathop{\mathrm{v}}}} \cU_{\text{v}} \cR_{\text{v}},
\end{equation}
where $\cC_{\ensuremath{\mathop{\mathrm{v}}}} \in \R^{r\times n_c}$ and $\cR_{\ensuremath{\mathop{\mathrm{v}}}} \in \R^{n_r\times r^3}$ consist of columns and rows of $\cVe$, respectively. Let us assume that $\cI_{\textbf{R}} \subseteq \{1,\ldots,n\}$ denotes the indices, leading to the construction of the matrix $\cR_{\text{v}}$ in \eqref{eq:CUR_Veff}, i.e.,
$$\cR_{\text{v}} = \cVe(\cI_{\textbf{R}},:).$$
As a result, we use the relation \eqref{eq:CUR_Veff} in \eqref{eq:H3_computation} to obtain
\begin{align*}
 \We^T {\cVe} \hx^{\circled{\tiny{$3$}}} &\approx \We^T \cC_{\ensuremath{\mathop{\mathrm{v}}}} \cU_{\text{v}} \cR_{\text{v}} \hx^{\circled{\tiny{$3$}}} \\
 &\approx \underbrace{\We^T \cC_{\ensuremath{\mathop{\mathrm{v}}}} \cU_{\text{v}} }_{=:\Psi}  \left( \cVe(\cI_{\textbf{R}},:)\hx^{\circled{\tiny{$3$}}} \right).
\end{align*}

Now, it can be noticed that the term $\cVe(\cI_{\textbf{R}},:) \hx^{\circled{\tiny{$3$}}}$ is nothing but evaluating the nonlinearity (in this case, it is a cubic nonlinearity) at indices $\cI_{\textbf{R}}$. As a result, we need to determine the nonlinearity at $n_r$ points.  Moreover, the matrix $\Psi \in \R^{r\times r_v}$ can be precomputed. This is exactly the idea of  hyper-reduction methods such as (D)EIM proposed in \cite{morBarMNetal04,morChaS10} in the case of nonlinear MOR. However, a major difference between the methodology in this paper and (D)EIM is that we do not require time-domain snapshots of the nonlinearity as needed in the case of DEIM, but we rather approximate the nonlinear terms in the reduced-order systems.  Summarizing, for the Chafee-Infante equation in the end, we have 
\begin{align}
 \We^T {\cVe} \left(\hx\otimes \hx\otimes \hx\right) &\approx \Psi \left(\tx\circ \tx \circ \tx\right),
\end{align}
where $\tx = \Ve(\cI_\textbf{R},:)\hx$, which is of complexity $\cO(r\cdot n_r^2 )$.
\begin{remark}
 In the above, we have focused on the computational aspect related to $\hH_\xi$ and $\hH_\xi\hx^{\circled{\tiny{$\xi$}}}$. However, analogously,  a complexity reduction can be performed for computing $\hN_\eta$ and $\hN_\eta\hx^{\circled{\tiny{$\eta$}}}$.
\end{remark}

\section{Parametric Polynomial Systems}
Until now, we have discussed the non-parametric case, i.e., all the system matrices are assumed to be constant. In this section, we briefly present an extension of the idea presented in the previous sections to parametric polynomial system \eqref{eq:para_poly_sys}. Similar to the non-parametric case, we can derive generalized transfer functions for the parametric case, which are given as follows:
\begin{subequations}\label{eq:Parageneral_TF}
\begin{align}
 \bF_L(s_1,\bp) &=  C(\bp)\Phi(s_1,\bp)B(\bp),\\
 \bF^{(\xi)}_H(s_1,\ldots,s_{\xi+1},\bp) & =  C(\bp)\Phi(s_{\xi+1},\bp) H_\xi \left(\Phi(s_\xi,\bp) \otimes \cdots \otimes \Phi(s_{1},\bp)\right)B(\bp)^{\circled{\tiny{$\xi$}}},\\
 \bF_N^{(\eta)}(s_1,\ldots, s_{\eta+1},\bp) & =  C(\bp) \Phi(s_{\eta+1},\bp) N_\eta\left(I_m\otimes\Phi(s_{\eta},\bp) \otimes \cdots \otimes \Phi(s_{1},\bp) \right) B(\bp)^{\circled{\tiny{$\eta$}}},
\end{align}
\end{subequations}
where $\Phi(s,\bp) = (sE(\bp)-A(\bp))^{-1}$. In the following, we present an extension of \cref{lemma:singleInterpolation} to the parametric case. 
\begin{theorem}\label{thm:parametric_interpolation}
	Consider the original parametric polynomial system as given in \eqref{eq:poly_sys}. Let $\sigma_i,\bp_i$, $i \in \{1,\ldots,\tr\}$ be interpolation points such that $sE(\bp)-A(\bp)$ is invertible for all $s\in \{\sigma_1,\ldots, \sigma_{\tr}\}$, $\bp \in \{\bp_1,\ldots, \bp_{\tr}\}$, and let $b_i \in \Rm$ and $c_i \in \Rq$ $i \in \{1,\ldots,\tr\}$ be tangential directions. Moreover, let $V$ and $W$ be defined as follows: 
\begin{subequations}\allowdisplaybreaks
	\begin{align*}
	V_{L}&= \bigcup_{i= 1}^{\tr} \range{\Phi(\sigma_i,\bp_i)B(\bp_i)b_{i}},\\
	V_{N} &= \bigcup_{\eta= 1}^d\bigcup_{i= 1}^{\tr} \range{ \Phi(\sigma_i,\bp_i) N_\eta\left(I_m\otimes \Phi(\sigma_i,\bp_i)B(\bp_i)b_i \otimes \cdots \otimes \Phi(\sigma_i,\bp_i)B(\bp_i)b_i \right)},  \\
	V_H &=  \bigcup_{\xi= 2}^d\bigcup_{i= 1}^{\tr}  \range{ \Phi(\sigma_i,\bp_i) H_\xi \left(\Phi(\sigma_i,\bp_i)B(\bp_i)b_i \otimes \cdots \otimes \Phi(\sigma_i,\bp_i)B(\bp_i)b_i \right) },\\ 
	W_L&= \bigcup_{i= 1}^{\tr}\range{\Phi(\sigma_i,\bp_i)^TC(\bp_i)^Tc_i}\\
	W_{N} &= \bigcup_{\eta= 1}^d \bigcup_{i= 1}^{\tr}\ensuremath{\mathop{\mathrm{range}}\left( \Phi(\sigma_{i},\bp_i) \left(N_\eta \right)_{(2)}\left(I_m\otimes\Phi(\sigma_i,\bp_i)B(\bp_i)b_{i} \otimes \cdots \right.\right.}\\
	&\hspace{5cm}  \otimes \Phi(\sigma_i,\bp_i)B(\bp_i)b_{i} \otimes\Phi(\sigma_i,\bp_i)^TC(\bp_i)^Tc_{i} \Big), \\
	W_{H} &= \bigcup_{\xi= 2}^d \bigcup_{i= 1}^{\tr}\ensuremath{\mathop{\mathrm{range}}\left(\Phi(\sigma_{i},\bp_i) \left(H_\xi\right)_{(2)} \left(\Phi(\sigma_{i},\bp_i)B(\bp_i)b_{i} \otimes \cdots \right.\right.}\\
	&\hspace{5cm}  \otimes \Phi(\sigma_{i},\bp_i)B(\bp_i)b_{i}\otimes \Phi(\sigma_i,\bp_i)^TC(\bp_i)^Tc_{i} \Big), \\
	V &= \range{V_L, V_N, V_H},\\
	W&= \range{W_L, W_N, W_H}.
	\end{align*}
\end{subequations}
If a reduced-order system is computed using the projection matrices $V$ and $W$ given above, assuming $V$ and $W$ are full rank matrices, as follows:
\begin{equation}\label{eq:Paracompute_rom}
\begin{aligned}
	\hE(\bp) &= W^TE(\bp)V, & \hA(\bp)& = W^TA(\bp)V, & \hH_{\xi}(\bp)& = W^TH_{\xi}(\bp)V^{\circled{\tiny{$\xi$}}},\quad \xi \in \{2,\ldots,d\},\\
	\hB(\bp) &= W^TB(\bp), & \hC(\bp)& = C(\bp)V, & \hN_{\eta}(\bp)& = W^TN_{\eta}(\bp)V^{\circled{\tiny{$\eta$}}},\quad \eta\in \{1,\ldots,d\}.
\end{aligned}
\end{equation}
then the following interpolation conditions are fulfilled:
	\begin{subequations}\allowdisplaybreaks
		\begin{align*}
		\bF_L(\sigma_i,\bp_{i}) b_i&= \hat{\bF}_L(\sigma_i,\bp_{i})b_i, \\
		c_i^T\bF_L(\sigma_i,\bp_{i})&= c_i^T\hat{\bF}_L(\sigma_i,\bp_{i}), \\
		\dfrac{d}{ds_1} c_i^T\bF_L(\sigma_i,\bp_{i})b_i&= \dfrac{d}{s_1} c_i^T\hat{\bF}_L(\sigma_i,\bp_{i})b_i,  \\
		\dfrac{d}{d\bp} c_i^T\bF_L(\sigma_i,\bp_{i})b_i&= \dfrac{d}{d\bp} c_i^T\hat{\bF}_L(\sigma_i,\bp_{i})b_i,  \\
		\bF_N^{(\eta)}(\sigma_i,\ldots,\sigma_i,\bp_{i})\left(I_m\otimes b_i^{\circled{\tiny {$\eta$}}}\right) &= 	\hat F_N^{(\eta)}(\sigma_i,\ldots,\sigma_i,\bp_{i}) \left(I_m\otimes b_i^{\circled{\tiny {$\eta$}}}\right), \\
		c_i^T\bF_N^{(\eta)}(\sigma_i,\ldots,\sigma_i,\bp_{i})\left( I_m^{\circled{\tiny {$2$}}} \otimes  b_i^{\circled{\tiny {$\eta{-}1$}}}\right) &= 	c_i^T\hat{\bF}_N^{(\eta)}(\sigma_i,\ldots,\sigma_i,\bp_{i})\left( I_m^{\circled{\tiny {$2$}}} \otimes  b_i^{\circled{\tiny {$\eta{-}1$}}}\right)\\
		\dfrac{d}{ds_j}c_i^T\bF_N^{(\eta)}(\sigma_i,\ldots,\sigma_i,\bp_{i})\left(I_m\otimes b_i^{\circled{\tiny {$\eta$}}}\right) &= 	\dfrac{d}{ds_j} \hat c_i^T\hat{\bF}_N^{(\eta)}(\sigma_i,\ldots,\sigma_i,\bp_{i}) \left(I_m\otimes b_i^{\circled{\tiny {$\eta$}}}\right), \\
		\dfrac{d}{d\bp}c_i^T\bF_N^{(\eta)}(\sigma_i,\ldots,\sigma_i,\bp_{i})\left(I_m\otimes b_i^{\circled{\tiny {$\eta$}}}\right) &= 	\dfrac{d}{d\bp} \hat c_i^T\hat{\bF}_N^{(\eta)}(\sigma_i,\ldots,\sigma_i,\bp_{i}) \left(I_m\otimes b_i^{\circled{\tiny {$\eta$}}}\right), \\
		\bF_H^{(\xi)}(\sigma_i,\ldots,\sigma_i,\bp_{i}) b_i^{\circled{\tiny {$\xi$}}} &= 	\hat{\bF}_H^{(\xi)}(\sigma_i,\ldots,\sigma_i,\bp_{i})  b_i^{\circled{\tiny {$\xi$}}}, \\
		c_i^T\bF_H^{(\xi)}(\sigma_i,\ldots,\sigma_i,\bp_{i})\left( I_m \otimes  b_i^{\circled{\tiny {$\xi{-}1$}}}\right) &= 	c_i^T\hat{\bF}_H^{(\xi)}(\sigma_i,\ldots,\sigma_i,\bp_{i})\left( I_m \otimes  b_i^{\circled{\tiny {$\xi{-}1$}}}\right),\\
		\dfrac{d}{ds_j}c_i^T\bF_H^{(\xi)}(\sigma_i,\ldots,\sigma_i,\bp_{i}) b_i^{\circled{\tiny {$\xi$}}} &= 	\dfrac{d}{ds_j}  c_i^T\hat{\bF}_H^{(\xi)}(\sigma_i,\ldots,\sigma_i,\bp_{i}) b_i^{\circled{\tiny {$\xi$}}},\\
		\dfrac{d}{d\bp}c_i^T\bF_H^{(\xi)}(\sigma_i,\ldots,\sigma_i,\bp_{i}) b_i^{\circled{\tiny {$\xi$}}} &= 	\dfrac{d}{d\bp}  c_i^T\hat{\bF}_H^{(\xi)}(\sigma_i,\ldots,\sigma_i,\bp_{i}) b_i^{\circled{\tiny {$\xi$}}},
		\end{align*}
	\end{subequations}
	where $F_L$, $F_N^{(\eta)}$, $F_H^{(\xi)}$, $\hF_L$, $\hF_N^{(\eta)}$, $\hF_H^{(\xi)}$ are also assumed to be differentiable with respect to $s_j$  and $\bp$.
\end{theorem}
\begin{proof}
	The proof of the lemma can be proven along the lines of the proof of \Cref{thm:gen_interpolation}. Therefore, for the sake of brevity of the paper, we skip the proof. 
\end{proof}
In the above, we have assumed a general parametric structure for the system matrices, e.g., $E(\bp), A(\bp)$, and the corresponding reduced-order system can be computed as shown in \eqref{eq:Paracompute_rom}. However, if we assume an affine parametric structure of the system matrices as follows:
\begin{equation}\label{eq:affineSys}
\begin{aligned}
E(\bp) &= \sum_{i=1}^{t_e} \alpha^{(i)}_e(\bp) E^{(i)},  & A(\bp) &= \sum_{i=1}^{t_a} \alpha^{(i)}_a(\bp) A^{(i)}, & H_\xi(\bp) &= \sum_{i=1}^{t_{h_\xi}} \alpha^{(i)}_{h_\xi}(\bp) H_\xi^{(i)},\\
B(\bp) &= \sum_{i=1}^{t_b} \alpha^{(i)}_b(\bp) B^{(i)},  & C(\bp) &= \sum_{i=1}^{t_c} \alpha^{(i)}_c(\bp) C^{(i)}, &N_\eta(\bp) &= \sum_{i=1}^{t_{n_\eta}} \alpha^{(i)}_{n_\eta}(\bp) N_\eta^{(i)},
\end{aligned}
\end{equation}
then the resulting reduced-order system can be determined, having the same structure
\begin{equation}\label{eq:parameter_reduced}
\begin{aligned}
\hE(\bp) &= \sum_{i=1}^{t_e} \alpha^{(i)}_e(\bp) \hE^{(i)},  & \hA(\bp) &= \sum_{i=1}^{t_a} \alpha^{(i)}_a(\bp) \hA^{(i)}, & \hH_\xi(\bp) &= \sum_{i=1}^{t_{h_xi}} \alpha^{(i)}_{h_\xi}(\bp) \hH_\xi^{(i)},\\
\hB(\bp) &= \sum_{i=1}^{t_b} \alpha^{(i)}_b(\bp) \hB^{(i)},  & \hC(\bp) &= \sum_{i=1}^{t_c} \alpha^{(i)}_c(\bp) \hC^{(i)}, &\hN_\eta(\bp) &= \sum_{i=1}^{t_{n_\eta}} \alpha^{(i)}_{n_\eta}(\bp) \hN_\eta^{(i)},
\end{aligned}
\end{equation}
where the original matrices are projected by the standard projection for given projection matrices; for example:
\begin{equation}
\begin{aligned}
\hE^{(1)} &= W^TE^{(1)}V, & \hA^{(1)}& = W^TA^{(1)}V, & \hH_{\xi}^{(1)}& = W^TH_{\xi}^{(1)}V^{\circled{\tiny{$\xi$}}},\quad \xi \in \{2,\ldots,d\},\\
\hB^{(1)} &= W^TB^{(1)}, & \hC^{(1)}& = C^{(1)}V, & \hN_{\eta}^{(1)}& = W^TN_{\eta}^{(1)}V^{\circled{\tiny{$\eta$}}},\quad \eta\in \{1,\ldots,d\}.
\end{aligned}
\end{equation}

Furthermore, like in the non-parametric case, we can easily develop the connection to the Loewner and shifted-Loewner type system, assuming we have sufficient interpolation points for frequency and parameter variables. In \Cref{algo:para_MOR}, we outline all the steps to construct reduced-order systems for the parametric case, which is again inspired by the Loewner-type MOR.

\begin{algorithm}[!tb]
	\caption{Model Reduction for Parametric Polynomial Systems \eqref{eq:para_poly_sys}.}
	\label{algo:para_MOR}
	\begin{algorithmic}[1]
		\Statex {\bf Input:} The original system~\cref{eq:para_poly_sys} with the affine structure as in \eqref{eq:affineSys}, and a set of interpolation points for the frequency and parameters, i.e., $\sigma_i$ and~$\bp_i$, $i \in \{1,\dots,\tilde r\}$.
		\Statex {\bf Output:}  A reduced-order system given as in \cref{eq:parameter_reduced}. 
		\State Determine $V$ and $W$ as shown in \Cref{thm:parametric_interpolation}.
		\State Based on $E(\bp)$ and $A(\bp)$ in \cref{eq:affineSys}, define
$$
			\LL^{(i)} = W^TE^{(i)} V, \quad i \in \{1,\ldots, t_e\},\qquad
			\LL_s^{(i)} = W^TA^{(i)} V, \quad i \in \{1,\ldots, t_a\}
$$
\begin{align}
\bbm \LL^{(1)},\ldots,\LL^{(t_e)}, \LL^{(1)}_s,\ldots,\LL_s^{(t_a)} \ebm &= Y_1\Sigma_1 X^T_1\\
\bbm \LL^{(1)}\\ \vdots\\  \LL^{(t_e)} \\  \LL^{(1)}_s\\ \vdots\\ \LL_s^{(t_a)} \ebm &= Y_2\Sigma_2 X^T_2.
\end{align}
\State Define $Y_r := Y_1(:,1:r)$ and $X_r := X_2(:,1:r)$. 
\State Determine the compact projection matrices:
\Statex~$V := \texttt{orth}{\left(VX_r\right)} $ and $W := \texttt{orth}{\left(WY_r\right)}$.
\State Determine a reduced-order system as shown in \eqref{eq:parameter_reduced}.
	\end{algorithmic}
\end{algorithm}

\section{Numerical Results}\label{sec:Numerics}
In this section, we illustrate the efficiency of the proposed methods by means of two nonlinear PDE examples and their variants.  All the simulations were done
on an \intel~\coreiseven-6700 CPU@3.40GHz, 8 MB cache, 8 GB RAM, Ubuntu 16.04, \matlab~Version 9.1.0.441655(R2016b) 64-bit(glnxa64). In the following, we note some details used in the numerical simulations:
\begin{itemize}
\item All original and reduced-order systems are integrated by the routine \texttt{ode15s}
in \matlab~with  relative error and absolute error tolerances of $10^{-10}$.
\item  We measure the output at 500 equidistant points within the time interval
$[0, T]$, where $T$ is the end time. 
\item We choose interpolation points for the frequency $(s)$ in a logarithmic scale  for a given frequency range, and   interpolation points for parameters $(\bp)$ are chosen randomly using the \texttt{rand} command. To ensure reproducibility, we use \texttt{randn(`seed',0)} to initialize a random number generator. 
\end{itemize}
\subsection{Chafee-Infante equation}\label{subsec:chafee_nonpara}
In our first example,  we deal with a widely considered one-dimensional Chafee-Infante equation. Its governing equation and boundary conditions are given as follows:
\begin{equation}\label{eq:chafeeEqn}
\begin{aligned}
\dot{v}(t) &=v_{xx} +v(1-v^2), \quad x \in (0,L)\times (0,T), & v(0,\cdot)& = u(t), ~~ (0,T),\\
v_x(L,\cdot) &=0, ~~(0,T), & v(x,0) &=0,~~(0,L).
\end{aligned}
\end{equation}
MOR of this example has been considered in various papers \cite{morBenB15,morBenG17,morBenGG18,gosea2018data}, where the authors have proposed different methods to reduce it. The governing equation has cubic nonlinearity. In the literature, a common approach to reduce such a cubic nonlinear system via system-theoretic MOR is twofold. First, it is to rewrite the cubic system into a QB system by introducing auxiliary variables. Thereafter, one can reduce it by employing a MOR scheme for QB systems such as balanced truncation \cite{morBenG17}, and interpolation based approaches, e.g., \cite{morAhmBJ16,morBenB15,morBenGG18}. However, in this process, we lose the original cubic nonlinearity structure in the reduced-order system.  On the other hand, the proposed method, in this paper, allows us to reduce a cubic system directly, having preserved the nonlinearity in the reduced-order system. 

We set the domain length $L = 1$. The system of equations~\cref{eq:chafeeEqn} is discretized using a finite-difference method by taking $k = 500$ grid points. Next, we aim at constructing a reduced cubic system by applying \Cref{algo:nonParaPoly}. For this purpose,  we consider the frequency range $\left[10^{-3},10^3\right]$.  For comparison, we also rewrite the cubic system into the QB form, which results in an equivalent QB system of order $n_{qb} = 2\cdot 500 =1000$. We consider the same frequency range in order to employ \cref{algo:nonParaPoly} to construct a reduced QB system.

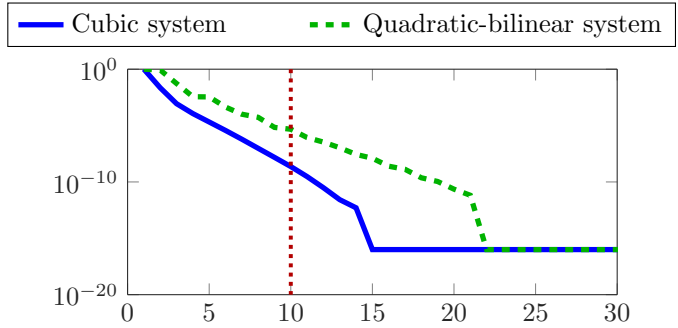
\begin{figure}[!tb]
  \centering
  \begin{tikzpicture}
      \begin{customlegend}[legend columns=-1, legend style={/tikz/every even column/.append style={column sep=1cm}} , legend entries={Cubic system, Quadratic-bilinear system }, ]
      \addlegendimage{blue, line width =2pt}
      \addlegendimage{black!30!green,dashed,line width = 2pt}
      \end{customlegend}
  \end{tikzpicture}
  \setlength\fheight{3.0cm}
  \setlength\fwidth{.5\textwidth}
  \includetikz{ChafeeDecaySingularValues}
  \caption{Relative decay of singular values based on the Loewner pencils, obtained using the original cubic system and its equivalent transformed QB system. }
  \label{figure:chafee_SVD}
\end{figure}

First, in \Cref{figure:chafee_SVD}, we observe the decay of the singular values, obtained from the Loewner pencil ($\LL-s\LL_s$).  As one can expect, the singular values related to the original cubic system decay faster as compared to its equivalent QB form. Hence, for the same order of the reduced-order system, we can anticipate a better quality reduced system.  Next, we construct the reduced cubic and QB systems of order $r = 10$. To test the quality of both reduced cubic and QB systems, we perform time-domain simulations using control inputs $u^{(1)}(t) = 10(\sin(\pi t)+1)$ and $u^{(2)}(t) = 5\left(te^{-t}\right)$, which are compared in \Cref{fig:chafee_input1,fig:chafee_input2} by showing the responses and relative errors. As can be seen from these figures, the cubic reduced-order system captures the dynamics of the original system much better as compared to the QB system; precisely, we gain  up to 3 orders of magnitude better accuracy using the new method.

\begin{figure}[!tb]
    \centering
      \begin{tikzpicture}
    \begin{customlegend}[legend columns=-1, legend style={/tikz/every even column/.append style={column sep=.5cm}} , legend entries={Original system, Cubic system, Quadratic-bilinear system}, ]
    \addlegendimage{color=black!30!red,solid,line width=2pt,mark = o}
    \addlegendimage{color=blue,line width=2pt}
    \addlegendimage{color=black!30!green,line width=2pt,dashed}
    \end{customlegend}
    \end{tikzpicture}    
    \begin{subfigure}[t]{.5\textwidth}
        \centering
        \setlength\fheight{2.5cm}
	\setlength\fwidth{.8\textwidth}
	\includetikz{ChafeeInput1_Response}        
	\caption{Transient response.}
	\label{fig:chafee_input1_res}
    \end{subfigure}%
    \begin{subfigure}[t]{.5\textwidth}
        \centering
        \setlength\fheight{2.5cm}
	\setlength\fwidth{.8\textwidth}
	\includetikz{ChafeeInput1_RelErr}
	\caption{Relative error.}
	\label{fig:chafee_input1_err}
    \end{subfigure}
    \caption{Chafee-Infante equation: a comparison of the original and reduced-order systems for the input $u^{(1)} = 10\left(\sin(\pi t)+1\right)$.}
    \label{fig:chafee_input1}
\end{figure}

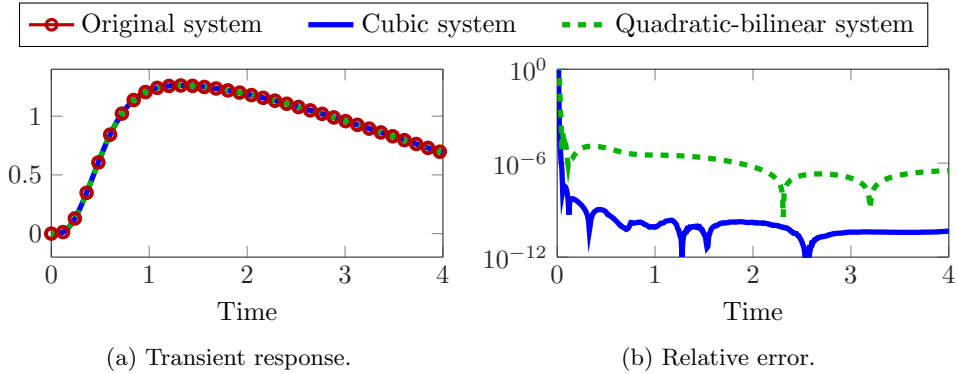
\begin{figure}[!tb]
    \centering
      \begin{tikzpicture}
    \begin{customlegend}[legend columns=-1, legend style={/tikz/every even column/.append style={column sep=.5cm}} , legend entries={Original system, Cubic system, Quadratic-bilinear system}, ]
    \addlegendimage{color=black!30!red,solid,line width=1.2pt,mark = o}
    \addlegendimage{color=blue,line width=2pt}
    \addlegendimage{color=black!30!green,line width=2pt,dashed}
    \end{customlegend}
    \end{tikzpicture}    
    \begin{subfigure}[t]{.5\textwidth}
        \centering
        \setlength\fheight{2.5cm}
	\setlength\fwidth{.8\textwidth}
	\includetikz{ChafeeInput2_Response}        
	\caption{Transient response.}
	\label{fig:chafee_input2_res}
    \end{subfigure}%
    \begin{subfigure}[t]{.5\textwidth}
        \centering
        \setlength\fheight{2.5cm}
	\setlength\fwidth{.8\textwidth}
	\includetikz{ChafeeInput2_RelErr}
	\caption{Relative error.}
	\label{fig:chafee_input2_err}
    \end{subfigure}
    \caption{Chafee-Infante equation: a comparison of the original and reduced-order systems for the input $u^{(2)} = 5\left(te^{-t}\right)$.}
    \label{fig:chafee_input2}
\end{figure}

\subsection{The FitzHugh-Nagumo(FHN) model}
As a second  non-parametric example, we consider the FHN system, which describes basic neuronal dynamics. This is a coupled cubic nonlinear system, whose governing equations and boundary conditions are as follows:
\begin{equation}
\begin{aligned}
    \epsilon v_t &=  \epsilon^2 v_{xx} + v(v-0.1)(1-v) -w + q,\\
    w_t &= hv - \gamma w + q\\
 \end{aligned}
\end{equation}
with boundary conditions
\begin{align*}
 v(x,0) &= 0,& w(x,0)& = 0, \qquad x\in (0,L),\\
 v_x(0,t) &= i_0(t), & v_x(1,t) &= 0,\qquad t\geq 0,
\end{align*}
where  $h = 0.05, \gamma = 2, q = 0.05, L = 0.1$ and $i_0$ acts an actuating control input which takes the values $5\cdot 10^4 t^3e^{-t}$, and briefly mentioning, the variables $v$ and $w$ denote the activation and de-activation of a neuron, respectively. We discretize the governing equation using a finite difference method, having taken $100$ grid points. This leads to a cubic system of order $n = 200$. We use the same output setting as used, e.g., in \cite{morBenGG18}. The system has two inputs and two outputs, thus, is a MIMO system. The MOR problem related to the FHN system has been considered by several researchers, see, e.g., \cite{morBenG17,morBenGG18,morChaS10}.  Similar to  the previous example, system-theoretic MOR of the FHN system  has also been considered by first rewriting it into a QB system and employing MOR  schemes such as interpolation-based and balanced truncation to reduce it. Thus, we obtain an equivalent QB system of order $n_{qb} = 300$. However, by doing so, we lose the original nonlinear structure. 


\begin{figure}[!tb]
  \centering
  \begin{tikzpicture}
      \begin{customlegend}[legend columns=2, legend style={/tikz/every even column/.append style={column sep=1cm}} , legend entries={Cubic sys. (two-sided), Cubic sys. (one-sided),QB sys. (two-sided), QB sys. (one-sided) }, ]
      \addlegendimage{blue, line width =2pt}
      \addlegendimage{blue, dashed,forget plot,every mark/.append style={solid, fill=red}, mark=otimes*,line width = 1.2pt}
      \addlegendimage{green!30!green,dashed,forget plot,every mark/.append style={solid, fill=gray}, mark=triangle*,line width = 2pt}
            \addlegendimage{green!30!green,dotted, line width =2pt}
     \end{customlegend}
  \end{tikzpicture}
  \setlength\fheight{3.0cm}
  \setlength\fwidth{.5\textwidth}
  \includetikz{NewFHNDecaySingularValues}
 \caption{FHN model: relative decay of singular values based on the Loewner pencils, obtained via one-sided and two-sided projection of the corresponding systems.}
  \label{figure:FHN_SVD}
\end{figure}
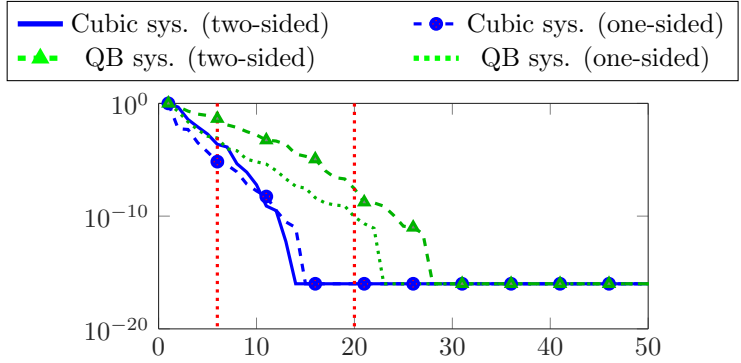

In order to apply \Cref{algo:nonParaPoly} to obtain reduced-order systems for the original cubic and its equivalent QB systems, we choose $200$ points in the frequency range $\left[10^{-2}, 10^2\right]$. In \Cref{figure:FHN_SVD}, we first show the relative decay of the singular values for the Loewner pencils (denoted by cubic sys.\ (two-sided) and QB sys.\ (two-sided)).  Therein, as expected, we observe that the singular values of the cubic system decay faster as compared to its equivalent QB system. Next, we construct reduced cubic and QB systems of order $r = 20$. To determine the quality of the reduced systems, we perform time-domain simulations and plot the result in \Cref{fig:FHN_input1}. We observe that the obtained cubic reduced system captures the dynamics of the original system very well, whereas the reduced QB system is unstable. This illustrates a common shortcoming of  \Cref{algo:nonParaPoly} (the Loewner approach) that it does not always result in a stable reduced system. 

As a remedy, we propose to obtain a reduced-order system using Galerkin (one-sided) projection.  For this, we determine the matrix $V$ at Step~1 in \Cref{algo:nonParaPoly} and set $W = V$. This is followed by determining $X_r$ as shown in Step~4 of the algorithm and determine the projection matrix $\Ve$. Subsequently, we set $\We = \Ve$ and compute a reduced-order system. As a result, we have a reduced-order system by Galerkin projection instead of Petrov-Galerkin projection. An advantage of doing Galerkin projection  is the (local) stability of the reduced-order system in some cases. Next, we compute reduced systems of order $r = 20$ using the cubic and its equivalent QB form, using Galerkin projection. 

First, we observe the decay of singular values in \Cref{figure:FHN_SVD}, showing that for Galerkin projection as well, the decay is faster for cubic systems as compared to the equivalent QB systems. Furthermore, we compare the transient response of the reduced-order systems obtained from Galerkin projection in \Cref{fig:FHN_input1}, which shows that the cubic reduced-order system performs much better as compared to the QB reduced systems. Interestingly, we also observe that the reduced cubic systems, using Petrov-Galerkin and Galerkin projection, tend to perform equally good as the time progresses but in the beginning, the reduced system, obtained using Petrov-Galerkin projection, performs better.  

Furthermore, we mention that the same order of accuracy as the reduced QB system (one-sided) of order $r = 20$ can be obtained from the reduced cubic order $r = 6$ only. Surprisingly, we also observe that the typical limit-cycles behavior of the system can be captured by the reduced cubic system of order as low as $r = 2$. On the other hand, the reduced QB could not capture the typical limit cycles behavior below the order $r = 15$. This is a profound observation, which illustrates that keeping the original structure of nonlinearity can lead to much better reduced-order systems.

\definecolor{mycolor1}{rgb}{1.00000,0.00000,1.00000}%
\definecolor{mycolor2}{rgb}{0.00000,1.00000,1.00000}%
\definecolor{mycolor3}{rgb}{1.00000,1.00000,0.00000}%

\begin{figure}[!tb]
    \centering
      \begin{tikzpicture}
    \begin{customlegend}[legend columns=2, legend style={/tikz/every even column/.append style={column sep=.0cm}} , legend entries={Ori.\  sys.\ , Cubic sys.\ (two-sided $r = 20$), Cubic sys.\ (one-sided $r = 20$), QB sys.\ (one-sided $r = 20$), Cubic sys.\ (two-sided $r = 6$)}, ]
    \addlegendimage{color=black!30!red,solid,line width=1.2pt,mark = o}
    \addlegendimage{color=blue,line width=2pt}
        \addlegendimage{color=mycolor2,solid,line width=2pt,mark = square*}
    \addlegendimage{color=black!30!green,line width=2pt,dashed, mark = triangle*,every mark/.append style={solid}}
    \addlegendimage{color=mycolor3!30!black,dotted, mark = +, every mark/.append style={solid}, line width = 1pt}
    \end{customlegend}
    \end{tikzpicture}    
    \begin{subfigure}[t]{.4\textwidth}
        \centering
        \setlength\fheight{2.5cm}
	\setlength\fwidth{.7\textwidth}
	\includetikz{NewFHNInput1_Response}        
	\caption{Transient response.}
	\label{fig:FHN_input1_res}
    \end{subfigure}%
    \begin{subfigure}[t]{.4\textwidth}
        \centering
        \setlength\fheight{2.5cm}
	\setlength\fwidth{.7\textwidth}
	\includetikz{NewFHNInput1_RelErr}
	\caption{Relative error.}
	\label{fig:FHN_input1_err}
    \end{subfigure}
    \caption{FHN model: a comparison of the original and reduced cubic and QB systems using one-sided and two-sided projections, having employed \cref{algo:nonParaPoly}.}
    \label{fig:FHN_input1}
\end{figure}
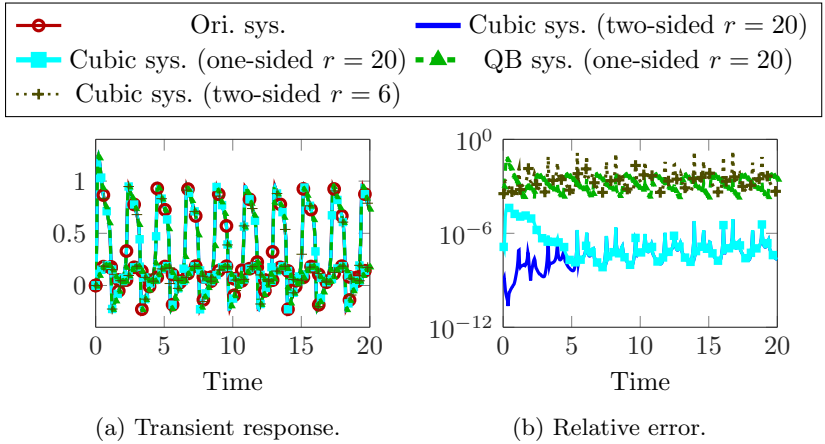

\subsection{Parametric Chafee-Infante equation}
As our parametric example, we consider the following parametric Chafee-Infante equation:
\begin{equation}
 \dot{v}(t) =v_{xx} +v(\bp-v^2),\quad x\in (0,L)\times (0,T),\\
\end{equation}
where $\bp \in \left[0.25,2\right]$. The boundary and initial conditions are the same as given in \Cref{subsec:chafee_nonpara}, and the domain length and discretization scheme are also chosen the same as in \Cref{subsec:chafee_nonpara}.  Next, we aim at constructing a reduced cubic parametric system using \Cref{algo:para_MOR}. For this, we take $200$ points in the frequency range $\left[10^{-3},10^3\right]$ and the equal number of points for the parameter in the considered interval. 

\begin{figure}[!tb]
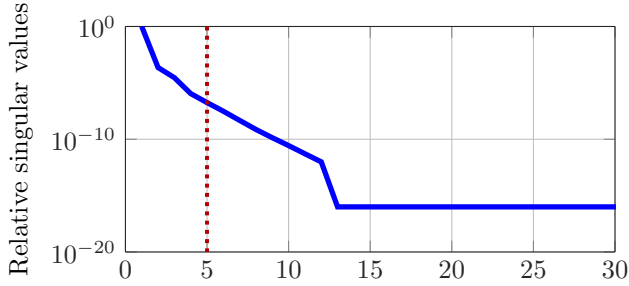

  \centering
  \setlength\fheight{3.0cm}
  \setlength\fwidth{.5\textwidth}
  \includetikz{ChafeeDecaSingularValuesPara}
  \caption{Parametric Chafee-Infante equation: relative decay of singular values based on the Loewner pencil.}
  \label{figure:chafee_SVD_Para}
\end{figure}

First, in \Cref{figure:chafee_SVD_Para}, we plot the decay of the singular values based on the Loewner pencil, which decays exponentially.  Subsequently, we determine a reduced parametric system of order $r = 5$. To compare the quality of the reduced-order system, we simulate for the same inputs as used in \Cref{subsec:chafee_nonpara} and for $\bp = \{0.25,1,2\}$. We plot the transient response and relative errors in \Cref{fig:chafee_input1_para,fig:chafee_input2_Para}, illustrating that the reduced parametric system can capture the dynamics of the system for different inputs and different parameter.

\begin{figure}[!tb]
    \centering
      \begin{tikzpicture}
    \begin{customlegend}[legend columns=3, legend style={/tikz/every even column/.append style={column sep=0.2cm}} , legend entries={Ori. sys. $(p = 0.25)$,Ori. sys. $(p = 1.0)$,Ori. sys. $(p = 2.0)$,Red. sys. $(p = 0.25)$, Red. sys. $(p = 1.0)$,  Red. sys. $(p = 2.0)$}, ]
    \addlegendimage{color=black!30!red,solid,line width=1.2pt,mark = triangle}
    \addlegendimage{color=black!30!red,solid,line width=1.2pt,mark = diamond}
    \addlegendimage{color=black!30!red,solid,line width=1.2pt,mark = otimes}
    \addlegendimage{color=blue,solid,line width=1.2pt,mark = triangle*,dashed}
    \addlegendimage{color=blue,solid,line width=1.2pt,mark = diamond*,dashed}
    \addlegendimage{color=blue,solid,line width=1.2pt,mark = otimes*,dashed}    \end{customlegend}
    \end{tikzpicture}    
    \begin{subfigure}[t]{.5\textwidth}
        \centering
        \setlength\fheight{2.5cm}
	\setlength\fwidth{.8\textwidth}
	\includetikz{ChafeeInput1_Response_Para}        
	\caption{Transient response.}
	\label{fig:chafee_input1_res_Para}
    \end{subfigure}%
    \begin{subfigure}[t]{.5\textwidth}
        \centering
        \setlength\fheight{2.5cm}
	\setlength\fwidth{.8\textwidth}
	\includetikz{ChafeeInput1_RelErr_Para}
	\caption{Relative error.}
	\label{fig:chafee_input1_err_Para}
    \end{subfigure}
    \caption{Parametric Chafee-Infante equation: a comparison of the original and reduced-order systems for the input $u^{(1)} = 10\left(\sin(\pi t)+1\right)$ and for different parameter values.}
    \label{fig:chafee_input1_para}
\end{figure}

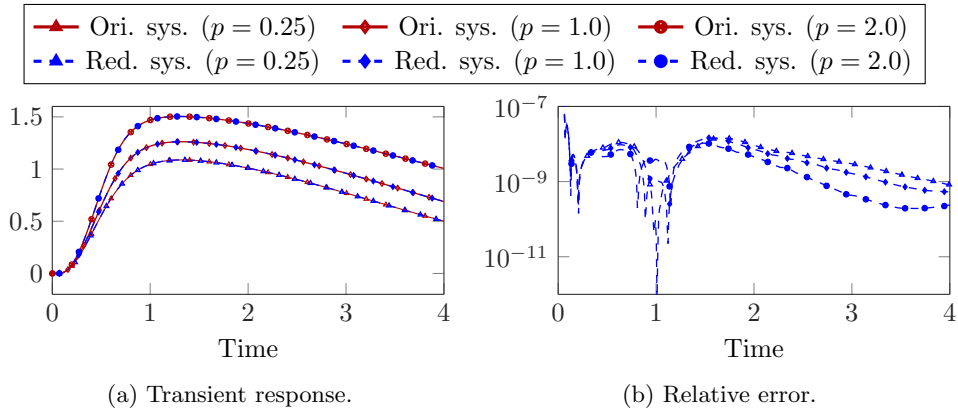
\begin{figure}[!tb]
    \centering
           \begin{tikzpicture}
    \begin{customlegend}[legend columns=3, legend style={/tikz/every even column/.append style={column sep=0.2cm}} , legend entries={Ori. sys. $(p = 0.25)$,Ori. sys. $(p = 1.0)$,Ori. sys. $(p = 2.0)$,Red. sys. $(p = 0.25)$, Red. sys. $(p = 1.0)$,  Red. sys. $(p = 2.0)$}, ]
    \addlegendimage{color=black!30!red,solid,line width=1.2pt,mark = triangle,every mark/.append style={solid},}
    \addlegendimage{color=black!30!red,solid,line width=1.2pt,mark = diamond,every mark/.append style={solid},}
    \addlegendimage{color=black!30!red,solid,line width=1.2pt,mark = otimes,every mark/.append style={solid},}
    \addlegendimage{color=blue,solid,line width=1.2pt,mark = triangle*,dashed,every mark/.append style={solid},}
    \addlegendimage{color=blue,solid,line width=1.2pt,mark = diamond*,dashed,every mark/.append style={solid},}
    \addlegendimage{color=blue,solid,line width=1.2pt,mark = otimes*,dashed,every mark/.append style={solid},}    \end{customlegend}
    \end{tikzpicture}  
    \begin{subfigure}[t]{.5\textwidth}
        \centering
        \setlength\fheight{2.5cm}
	\setlength\fwidth{.8\textwidth}
	\includetikz{ChafeeInput2_Response_Para}        
	\caption{Transient response.}
	\label{fig:chafee_input2_res_Para}
    \end{subfigure}%
    \begin{subfigure}[t]{.5\textwidth}
        \centering
        \setlength\fheight{2.5cm}
	\setlength\fwidth{.8\textwidth}
	\includetikz{ChafeeInput2_RelErr_Para}
	\caption{Relative error.}
	\label{fig:chafee_input2_err_Para}
    \end{subfigure}
    \caption{Parametric Chafee-Infante equation: a comparison of the original and reduced-order systems for the input $u^{(2)} = 5\left(e^{-t}t\right)$ and for different parameter values.}
    \label{fig:chafee_input2_Para}
\end{figure}

\subsection{Usage of CUR in ROM}
In this section, we illustrate the usage of the CUR matrix approximation to further approximate the nonlinear reduced terms. For this, we again consider the Chafee-Infante equation as considered in \Cref{subsec:chafee_nonpara}. Using the same setting as shown in \Cref{subsec:chafee_nonpara}, we aim at determining reduced cubic systems using one-sided and two-sided projections. First, in \Cref{figure:chafee_SVD_CUR}, we plot the relative decay of the singular values based on the Loewner pencil, obtained using the one-sided and two-sided projection matrices. We observe that the singular values based on the two-sided projection decay fast as compared to the one-sided projection. 

\begin{figure}[!tb]
  \centering
  \begin{tikzpicture}
      \begin{customlegend}[legend columns=-1, legend style={/tikz/every even column/.append style={column sep=1cm}} , legend entries={Two-sided projection, One-sided projection}, ]
      \addlegendimage{blue, line width =2pt}
      \addlegendimage{black!30!green,dashed,line width = 2pt}
      \end{customlegend}
  \end{tikzpicture}
  
  \setlength\fheight{3.0cm}
  \setlength\fwidth{.5\textwidth}
  \includetikz{ChafeeSingularDecay_CUR}
  \caption{Chafee-Infante equation: relative decay of singular values using the Loewner pencil, obtained via one-sided and two-sided projections.}
  \label{figure:chafee_SVD_CUR}
\end{figure}
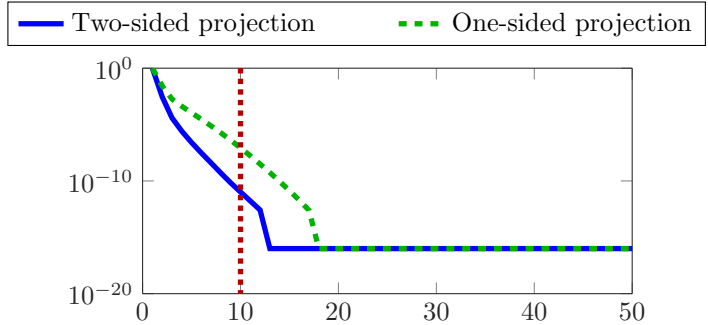

Next, we construct reduced-order systems of order $r = 10$ using one-sided and two-sided projections using \Cref{algo:nonParaPoly}, preserving the cubic nonlinear terms. As discussed in \Cref{sec:computationalCUR}, we can approximate these terms by making use of CUR matrix approximation. For CUR matrix approximation, we choose $60$ rows and $60$ columns of $\cV$ (defined in \cref{eq:H3_computation}), which are chosen based on an adaptive sampling proposed in \cite{wang2013improving}. We would like to mention that the number $60$ for row and columns is determined based on a trial and error method. An appropriate automatic method for CUR matrix approximation, being suitable for MOR, needs further research. To this end, we have four reduced systems as follows:
\begin{itemize}
 \item One-sided projection  (\texttt{OneSProj})
 \item One-sided projection, combined with CUR matrix approximation   \item[] \hfill (\texttt{OneSProj + CUR})
 \item Two-sided projection  (\texttt{TwoSProj})
 \item Two-sided projection, combined with CUR matrix approximation  \item[] \hfill  (\texttt{TwoSProj + CUR})
\end{itemize}
To compare the quality of these reduced-order systems, we perform the time-domain simulations of these systems with the original systems for two control inputs, the same as considered in \Cref{subsec:chafee_nonpara}. We observe that two-sided projection yields the best reduced-order systems among the four reduced-order systems. Furthermore, when the two-sided reduced-order system is combined with CUR matrix approximation, then we notice that the quality of the reduced-order system decreases a little but still provides a very good approximation of the original system. Interestingly, we also notice that CUR matrix approximation applied to the one-sided reduced-order system also performs really good, where the reduction in quality of \texttt{OsP} is very slim.

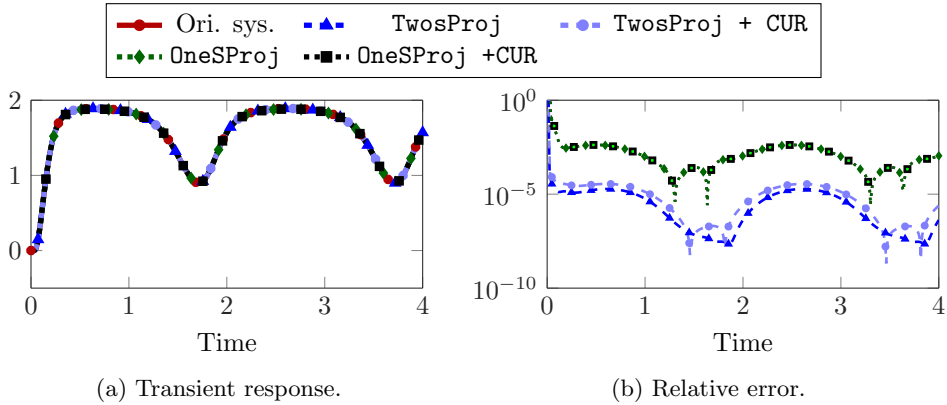
\begin{figure}[!tb]
    \centering
      \begin{tikzpicture}
    \begin{customlegend}[legend columns=3, legend style={/tikz/every even column/.append style={column sep=0.2cm}} , legend entries={Ori. sys.,\texttt{TwosProj},\texttt{TwosProj + CUR}, \texttt{OneSProj}, \texttt{OneSProj +CUR}}, ]
     \addlegendimage{color=black!30!red,solid,line width=2pt,mark = o,every mark/.append style={solid},mark size = 1.2pt}
    \addlegendimage{color=blue,dashed,line width=2pt,mark = triangle,every mark/.append style={solid},mark size = 1.2pt}
    \addlegendimage{color=blue!50,dashed,line width=2pt,mark = otimes*,every mark/.append style={solid},mark size = 1.2pt}
    \addlegendimage{color=black!60!green,dotted,line width=2pt,mark = diamond*,every mark/.append style={solid},mark size = 1.2pt}
    \addlegendimage{color=black,dotted,line width=2pt,mark = square*,every mark/.append style={solid},mark size = 1.2pt}
     \end{customlegend}
    \end{tikzpicture}    
    \begin{subfigure}[t]{.5\textwidth}
        \centering
        \setlength\fheight{2.5cm}
	\setlength\fwidth{.8\textwidth}
	\includetikz{ChafeeInput1_Response_CUR}        
	\caption{Transient response.}
	\label{fig:chafee_input1_res_CUR}
    \end{subfigure}%
    \begin{subfigure}[t]{.5\textwidth}
        \centering
        \setlength\fheight{2.5cm}
	\setlength\fwidth{.8\textwidth}
	\includetikz{ChafeeInput1_RelErr_CUR}
	\caption{Relative error.}
	\label{fig:chafee_input1_err_CUR}
    \end{subfigure}
    \caption{Chafee-Infante equation: a comparison of the original and (CUR combined) reduced-order systems for the input $u^{(1)} = 10\left(\sin(\pi t)+1\right)$.}
    \label{fig:chafee_input1_CUR}
\end{figure}

\begin{figure}[!tb]
    \centering
      \begin{tikzpicture}
    \begin{customlegend}[legend columns=3, legend style={/tikz/every even column/.append style={column sep=0.5cm}} , legend entries={Ori. sys.,\texttt{TwosProj},\texttt{TwosProj + CUR}, \texttt{OneSProj}, \texttt{OneSProj +CUR}}, ]
    \addlegendimage{color=black!30!red,solid,line width=2pt,mark = o,every mark/.append style={solid},mark size = 1.2pt}
    \addlegendimage{color=blue,dashed,line width=2pt,mark = triangle,every mark/.append style={solid},mark size = 1.2pt}
    \addlegendimage{color=blue!50,dashed,line width=2pt,mark = otimes*,every mark/.append style={solid},mark size = 1.2pt}
    \addlegendimage{color=black!60!green,dotted,line width=2pt,mark = diamond*,every mark/.append style={solid},mark size = 1.2pt}
    \addlegendimage{color=black,dotted,line width=2pt,mark = square*,every mark/.append style={solid},mark size = 1.2pt}
     \end{customlegend}
    \end{tikzpicture}   
    \begin{subfigure}[t]{.5\textwidth}
        \centering
        \setlength\fheight{2.5cm}
	\setlength\fwidth{.8\textwidth}
	\includetikz{ChafeeInput2_Response_CUR}        
	\caption{Transient response.}
	\label{fig:chafee_input2_res_CUR}
    \end{subfigure}%
    \begin{subfigure}[t]{.5\textwidth}
        \centering
        \setlength\fheight{2.5cm}
	\setlength\fwidth{.8\textwidth}
	\includetikz{ChafeeInput2_RelErr_CUR}
	\caption{Relative error.}
	\label{fig:chafee_input2_err_CUR}
    \end{subfigure}
    \caption{Chafee-Infante equation: a comparison of the original and (CUR combined) reduced-order systems for the input $u^{(2)} = 5\left(e^{-t}t\right)$.}
    \label{fig:chafee_input2_CUR}
\end{figure}
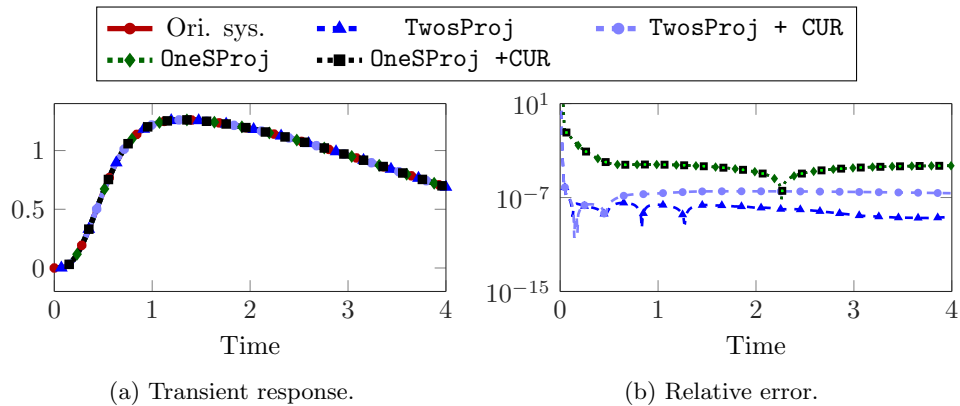

\section{Conclusions}
In this paper, we have discussed the construction of interpolating reduced-order systems for a parametric polynomial system. For this purpose, we have introduced generalized multi-variate  transfer functions for the systems. Furthermore, we have proposed algorithms, inspired by the Loewner approach, to generate good quality reduced-order systems in an automatic way. Furthermore, we have discussed related computational issues and also the usage of the CUR matrix approximation  in the simulations of reduced systems which may be helpful in some case. We have illustrated the efficiency of the approaches via several numerical experiments, where we have observed that preserving the original structure of the nonlinearity in reduced-order systems can lead to much better reduced-order systems. 

So far, in our numerical experiments, we have logarithmically or randomly chosen interpolation points  in  given intervals for frequency and parameters in order to apply \cref{algo:nonParaPoly,algo:para_MOR}. However, choosing these interpolation points wisely, instead of logarithmically or randomly, can ease the computational burden. In this direction, $\cH_2$-optimal framework \cite{morBenGG18} and an adaptive choice of interpolation points based on an error estimate \cite{morAhmadBF18} can be extended from quadratic-bilinear systems to polynomial systems.  Moreover, in \Cref{sec:computationalCUR}, we have discussed the computational aspect related to $\hH_{\xi}\kronF{\hx}{$\xi$} $ which can be eased with the help of CUR matrix approximation. However, we do not take the projection matrix $W$ into account for an approximation of the latter term. Thus, it would be valuable to employ the projection matrix $W$ as well, which could  improve the approximation quality. Also, an appropriate choice of CUR matrix approximation for MOR still demands some more investigation.  Last but not least, it will be of a great interest to the MOR community to extend the proposed methodology to other classes of nonlinear systems such as rational nonlinear systems, i.e., those systems containing nonlinear functions e.g., $\tfrac{1}{1+x}$ or $e^{-1/x}$. Such systems, for example, arise, in batch chromatography reactors \cite{guiochon2006fundamentals} and rector models \cite{kramer2018nonlinear}. However, we remark that such systems can be rewritten as a polynomial system by introducing auxiliary variables as discussed in \Cref{subsec:polynomialization}, but the goal would be to preserve the original structure of the nonlinearity in the reduced-order systems.

\bibliographystyle{siamplain}
\bibliography{mor}

\begin{thebibliography}{10}

\bibitem{morAhmadBF18}
{\sc M.~I. Ahmad, P.~Benner, and L.~Feng}, {\em Interpolatory model reduction
  for quadratic-bilinear systems using error estimators}, Engrg. Comput., 36
  (2018), pp.~25--44, \url{https://doi.org/10.1108/EC-04-2018-0162}.

\bibitem{morAhmBJ16}
{\sc M.~I. Ahmad, P.~Benner, and I.~Jaimoukha}, {\em Krylov subspace projection
  methods for model reduction of quadratic-bilinear systems}, IET Control
  Theory \& Applications, 10 (2016), pp.~2010--2018,
  \url{https://doi.org/10.1049/iet-cta.2016.0415}.

\bibitem{morAnt05}
{\sc A.~C. Antoulas}, {\em Approximation of Large-Scale Dynamical Systems},
  vol.~6 of Adv. Des. Control, {SIAM} Publications, Philadelphia, PA, 2005,
  \url{https://doi.org/10.1137/1.9780898718713}.

\bibitem{antoulas2016model}
{\sc A.~C. Antoulas, I.~V. Gosea, and A.~C. Ionita}, {\em Model reduction of
  bilinear systems in the {L}oewner framework}, {SIAM} J. Sci. Comput., 38
  (2016), pp.~B889--B916.

\bibitem{antoulas2014tutorial}
{\sc A.~C. Antoulas, S.~Lefteriu, and A.~C. Ionita}, {\em A tutorial
  introduction to the {L}oewner framework for model reduction}, in Model
  Reduction and Approximation: Theory and Algorithms, P.~Benner, A.~Cohen,
  M.~Ohlberger, and K.~Willcox, eds., SIAM, 2017, pp.~335--376,
  \url{https://doi.org/10.1137/1.9781611974829.ch8}.

\bibitem{morBarMNetal04}
{\sc M.~Barrault, Y.~Maday, N.~C. Nguyen, and A.~T. Patera}, {\em An `empirical
  interpolation' method: application to efficient reduced-basis discretization
  of partial differential equations}, C. R. Math. Acad. Sci. Paris, 339 (2004),
  pp.~667--672.

\bibitem{morBenB12b}
{\sc P.~Benner and T.~Breiten}, {\em Interpolation-based $\mathcal{H}_2$-model
  reduction of bilinear control systems}, {SIAM} J. Matrix Anal. Appl., 33
  (2012), pp.~859--885.

\bibitem{morBenB15}
{\sc P.~Benner and T.~Breiten}, {\em Two-sided projection methods for nonlinear
  model order reduction}, {SIAM} J. Sci. Comput., 37 (2015), pp.~B239--B260,
  \url{https://doi.org/10.1137/14097255X}.

\bibitem{morBenCOetal17}
{\sc P.~Benner, A.~Cohen, M.~Ohlberger, and K.~Willcox}, eds., {\em Model
  Reduction and Approximation: Theory and Algorithms}, Computational Science \&
  Engineering, SIAM Publications, Philadelphia, PA, 2017,
  \url{https://doi.org/10.1137/1.9781611974829}.

\bibitem{morBenG17}
{\sc P.~Benner and P.~Goyal}, {\em Balanced truncation model order reduction
  for quadratic-bilinear systems}, e-prints 1705.00160, arXiv, 2017,
  \url{https://arxiv.org/abs/1705.00160}.

\bibitem{morBenGG18}
{\sc P.~Benner, P.~Goyal, and S.~Gugercin}, {\em $\mathcal{H}_2$-quasi-optimal
  model order reduction for quadratic-bilinear control systems}, {SIAM} J.
  Matrix Anal. Appl., 39 (2018), pp.~983--1032,
  \url{https://doi.org/10.1137/16M1098280}.

\bibitem{morBenMS05}
{\sc P.~Benner, V.~Mehrmann, and D.~C. Sorensen}, {\em Dimension Reduction of
  Large-Scale Systems}, vol.~45 of Lect. Notes Comput. Sci. Eng.,
  Springer-Verlag, Berlin/Heidelberg, Germany, 2005.

\bibitem{morChaS10}
{\sc S.~Chaturantabut and D.~C. Sorensen}, {\em Nonlinear model reduction via
  discrete empirical interpolation}, {SIAM} J. Sci. Comput., 32 (2010),
  pp.~2737--2764, \url{https://doi.org/10.1137/090766498}.

\bibitem{morGalVV04}
{\sc K.~Gallivan, A.~Vandendorpe, and P.~Van~Dooren}, {\em Model reduction of
  {MIMO} systems via tangential interpolation}, {SIAM} J. Matrix Anal. Appl.,
  26 (2004), pp.~328--349.

\bibitem{gosea2018data}
{\sc I.~V. Gosea and A.~C. Antoulas}, {\em Data-driven model order reduction of
  quadratic-bilinear systems}, Numer. Lin. Alg. Appl., 25 (2018), p.~e2200,
  \url{https://doi.org/10.1002/nla.2200}.

\bibitem{morGu11}
{\sc C.~Gu}, {\em {QLMOR}: A projection-based nonlinear model order reduction
  approach using quadratic-linear representation of nonlinear systems}, IEEE
  Trans. Comput. Aided Des. Integr. Circuits. Syst., 30 (2011), pp.~1307--1320,
  \url{https://doi.org/10.1109/TCAD.2011.2142184}.

\bibitem{GubischV17mor}
{\sc M.~Gubisch and S.~Volkwein}, {\em Proper orthogonal decomposition for
  linear-quadratic optimal control}, in Model Reduction and Approximation:
  Theory and Algorithms, P.~Benner, A.~Cohen, M.~Ohlberger, and K.~Willcox,
  eds., SIAM, 2017, pp.~3--63,
  \url{https://doi.org/10.1137/1.9781611974829.ch1}.

\bibitem{morGugAB08}
{\sc S.~Gugercin, A.~C. Antoulas, and C.~Beattie}, {\em {$\mathcal{H}_2$} model
  reduction for large-scale linear dynamical systems}, {SIAM} J. Matrix Anal.
  Appl., 30 (2008), pp.~609--638, \url{https://doi.org/10.1137/060666123}.

\bibitem{guiochon2006fundamentals}
{\sc G.~Guiochon, A.~Felinger, and D.~G. Shirazi}, {\em Fundamentals of
  preparative and nonlinear chromatography}, Academic Press, Boston, 2006.

\bibitem{kolda2006multilinear}
{\sc T.~G. Kolda}, {\em Multilinear operators for higher-order decompositions},
  tech. report, Technical report SAND2006-2081, Sandia National Laboratories,
  Albuquerque, NM and Livermore, CA, 2006,
  \url{https://www.sandia.gov/~tgkolda/pubs/pubfiles/SAND2006-2081.pdf}.

\bibitem{kramer2018nonlinear}
{\sc B.~Kramer and K.~Willcox}, {\em Nonlinear model order reduction via
  lifting transformations and proper orthogonal decomposition}, AIAA J.,
  (2019), \url{https://doi.org/10.2514/1.J057791}.

\bibitem{mahoney2009cur}
{\sc M.~W. Mahoney and P.~Drineas}, {\em {CUR} matrix decompositions for
  improved data analysis}, Proc. of the National Academy of Sciences,  (2009),
  pp.~697--702.

\bibitem{mayo2007framework}
{\sc A.~J. Mayo and A.~C. Antoulas}, {\em A framework for the solution of the
  generalized realization problem}, Linear Algebra Appl., 425 (2007),
  pp.~634--662.

\bibitem{mccormick1976computability}
{\sc G.~P. McCormick}, {\em Computability of global solutions to factorable
  nonconvex programs:{Part I}--{C}onvex underestimating problems}, Mathematical
  programming, 10 (1976), pp.~147--175.

\bibitem{morQuaMN16}
{\sc A.~Quarteroni, A.~Manzoni, and F.~Negri}, {\em Reduced Basis Methods for
  Partial Differential Equations}, vol.~92 of La Matematica per il 3+2,
  Springer International Publishing, 2016.
\newblock ISBN: 978-3-319-15430-5.

\bibitem{rodriguez2018interpolatory}
{\sc A.~C. Rodriguez, S.~Gugercin, and J.~Borggaard}, {\em Interpolatory model
  reduction of parameterized bilinear dynamical systems}, Adv. Comput. Math.,
  44 (2018), pp.~1887--1916, \url{https://doi.org/10.1007/s10444-018-9611-y}.

\bibitem{rugh1981nonlinear}
{\sc W.~J. Rugh}, {\em Nonlinear system theory}, Johns Hopkins University Press
  Baltimore, 1981.

\bibitem{sorensen2016deim}
{\sc D.~C. Sorensen and M.~Embree}, {\em A {DEIM} induced {CUR} factorization},
  {SIAM} J. Sci. Comput., 38 (2016), pp.~A1454--A1482.

\bibitem{wang2013improving}
{\sc S.~Wang and Z.~Zhang}, {\em Improving {CUR} matrix decomposition and the
  {N}ystr{\"o}m approximation via adaptive sampling}, The Journal of Machine
  Learning Research, 14 (2013), pp.~2729--2769.

\end{thebibliography}
\end{document}